\documentclass[11pt,leqno,twoside]{article}

\usepackage[english]{babel}

\usepackage[a4paper,top=3cm,bottom=3cm,left=3cm,right=3cm,marginparwidth=1.75cm]{geometry}

\usepackage{amsmath,amsfonts,amsthm,amssymb}
\usepackage{graphicx}
\usepackage{tikz}
\usetikzlibrary{arrows.meta,calc,patterns}
\usepackage{pgfplots}
\usepgfplotslibrary{fillbetween}
\pgfplotsset{compat=1.18}
\usepackage{enumerate}
\usepackage[colorlinks=true, allcolors=blue]{hyperref}
\usepackage{natbib}
\usepackage{todonotes}
\usepackage{subcaption}

\numberwithin{equation}{section}

\newtheorem{remark}{Remark}
\newtheorem{definition}{Definition}
\newtheorem{prop}{Proposition}
\newtheorem{thrm}{Theorem}
\newtheorem{lemma}{Lemma}
\newtheorem*{assumptions}{Assumptions}

\DeclareMathOperator{\Mean}{\mathbf{E}}
\DeclareMathOperator{\Var}{\mathbf{Var}}
\DeclareMathOperator{\Prb}{\mathbf{P}}
\DeclareMathOperator{\Law}{Law}

\DeclareMathOperator{\argmin}{argmin}

\newcounter{hypcount}
\newenvironment{hypenv}{\begin{enumerate}\setcounter{enumi}{\value{hypcount}}}{\setcounter{hypcount}{\value{enumi}}\end{enumerate}}

\newcounter{hypcount2}

\newcommand{\hypref}[1]{(A\ref{#1})}
\newcommand{\hyprefall}{(A1)\,--\,(A\arabic{hypcount})}

\newcounter{dummy}
\newcommand{\hyprefallbutlast}{\setcounter{dummy}{\value{hypcount}}\addtocounter{dummy}{-1}(A1)\,--\,(A\arabic{dummy})}

\newcommand{\cA}{\mathcal{A}}
\newcommand{\cB}{\mathcal{B}}

\newcommand{\cF}{\mathcal{F}}

\newcommand{\cP}{\mathcal{P}}

\newcommand{\cT}{\mathcal{T}}

\newcommand{\cX}{\mathcal{X}}

\newcommand{\bF}{\mathbb{F}}
\newcommand{\bG}{\mathbb{G}}
\newcommand{\bS}{\mathbb{S}}

\newcommand{\bN}{\mathbb{N}}
\newcommand{\bR}{\mathbb{R}}

\renewcommand{\phi}{\varphi}
\renewcommand{\epsilon}{\varepsilon}

\title{Mean field games with option to buy information\footnote{
The authors thank Jos{\'e} Antonio Salmer{\'o}n Garrido for many helpful discussions.
}
}

\author{Bernardo D'Auria\footnote{Department of Mathematics ``Tullio Levi - Civita'', University of Padua, Italy, \texttt{dauria@math.unipd.it}}
\and Markus Fischer\footnote{Department of Mathematics ``Tullio Levi - Civita'', University of Padua, Italy, \texttt{fischer@math.unipd.it}}}

\date{November 4, 2026}

\begin{document}
\maketitle

\begin{abstract}
We introduce a class of continuous time finite horizon mean field games where the objective function of the representative player depends on a hidden state, in addition to position, control, and the population distribution. While acting on the position dynamics, the agent has the option to pay for seeing the hidden state. We connect the original formulation of our model with a mean field model of optimal control with discretionary stopping, characterize solutions, and give a couple of simple yet non-trivial explicitly solvable examples. For a class of $N$-player games with compatible information structure, we show that approximate Nash equilibria can be constructed starting from a solution to the limit model.
\end{abstract}

\textbf{Keywords and phrases:} mean field games, hidden state, information acquisition, control-stopping problem.\medskip

{\small \textbf{2020 AMS subject classifications:} \small \textbf{2020 AMS subject classifications:} 91A16, 91A27; 60G40, 93E20}

\section{Introduction}

Classical mean field games as introduced in \citet{huangetal06, lasrylions07} arise as limiting systems for symmetric, weakly interacting, non-zero-sum, non-cooperative stochastic $N$-player games as the number of players N tends to infinity. In this framework, a representative agent responds to the aggregate behavior of the population according to a prescribed objective function. An essential feature is the choice of admissible strategies, as it determines the information structure of the limit as well as prelimit games and the way they are connected; see \citet{lacker20} and the references therein. Different forms of partial information have recently been considered in the mean field game literature: trading with latent states and differing beliefs among subpopulations in \cite{casgrainjaimungal20}, partial observation of agent positions in a linear-quadratic setting in \cite{bensoussanyam21}, mean field games with unknown initial distribution in \citet{bertucci22+}, discrete time mean field games with controllable information delay in \citet{bechereretal23+}; a Stackelberg equilibrium between an informed major player and a minor player population in mean field game equilibrium in \cite{bergaultetal24}, an optimal stopping mean field game to estimate a hidden binary parameter in \cite{campbellzhang24+}, and a finite horizon mean field game with objective functional depending on a hidden Gaussian state with Bayesian learning from noisy signal in \citet{shmayayiliotto25}.

Here, we will study a simple continuous time finite horizon mean field game in which the objective function depends on a hidden state that is unknown at initial time. During the game, the representative agent has the option of paying to observe the hidden state. This cost can be interpreted as a measurement fee paid to eliminate model uncertainty. The cost functional depends on position, control, population distribution, and the hidden state, whereas the position dynamics is given by an initial condition, a directly controlled drift, and additive Wiener noise. Control actions are required to be  non-anticipative. The representative player is allowed to buy access to the hidden state at a \emph{stopping time} with respect to the filtration generated by the initial position and the driving Wiener process. Once acquired, the information on the hidden state will be available for control. The possibility of acquiring information divides the agent population into two groups: uninformed and informed agents. Mathematically, this decision can be described as a form of discretionary stopping. Compared to mean field games with control-stopping recently studied in the literature, see \citet{dumitrescuetal21, dumitrescuetal23} and the earlier work \citet{carmonaetal17} on mean field games of optimal stopping only, the stopping time in our model plays a different role as it does not terminate the game. Unlike in mean field games with absorption \citep[for instance,][]{campifischer18}, the stopping time here is directly controlled, namely through the decision to buy information. The study of optimal control-stopping problems, without game or mean field structure, has a long tradition; see, for instance, \citet{mazliak93}, \citet{pham98}, \citet{cecibassan04} and the more recent works by \citet{elasrietal22} and \cite{deangelismilazzo23}.

The rest of this paper is structured as follows. In Section~\ref{SectMFGstrong}, we specify the model and introduce a notion of mean field game equilibrium, called (strong) solution. In Section~\ref{SectMixedCP}, we reformulate the optimality condition in the definition of solutions in terms of an optimal control-stopping problem. Based on this connection, we heuristically derive, in Section~\ref{SectMFGSystem}, a coupled system of backward Hamilton-Jacobi-Bellman equations and a Kolmogorov forward equation. In Section~\ref{SectExample}, we discuss a simple linear-quadratic example, which also serves to illustrate existence of solutions. In Section~\ref{SectExampleMeanFieldBoundary}, we propose another example where the acquisition region is non-trivial and has an explicit form. Section~\ref{SectApproximation} introduces a class of $N$-player games with compatible information structure in the sense that solutions to our mean field game allow one to construct approximate $N$-player Nash equilibria.

\section{The mean field game in the strong formulation} \label{SectMFGstrong}

For a Polish space $\cX$, let $\cP(\cX)$ denote the space of probability measures on the Borel sets of $\cX$. Endow $\cP(\cX)$ with the topology of weak convergence of measures. Let $\bS$ denote the space of values for the hidden state. For simplicity, we assume that $\bS$ is a finite set endowed with the discrete topology. Some of our results can be extended to the case where $\bS$ is a general Polish space. Let $\mu \in \cP(\bS)$. The measure $\mu$ will thus be the commonly known distribution of the hidden state. Fix $s_0 \in \bS$, and define the multiplication of elements of $\bS$ with the real numbers $0$, $1$ by setting, for $s\in \bS$,  $s \cdot 1 \doteq s$, $s \cdot 0 \doteq s_{0}$.

The position of the representative player will evolve in $\bR^{d}$. Let $\nu \in \cP(\bR^d)$. The measure~$\nu$ will be the distribution of the initial position (at time zero) of the representative player. Throughout the paper, we assume that $\nu$ has finite exponential moments, that is, 
\[
\int \exp(c|x|) \, \nu(dx) < \infty \quad\text{for all }c> 0.
\]

Let $(\Omega,\cF, \Prb)$ be a complete probability space with $\bF=\{\cF_t\}_{t\ge0}$ a complete filtration in $\cF$ carrying a $d$-dimensional $\bF$-Wiener process $W$ starting in zero, an $\cF_0$-measurable $\bS$-valued random variable $S$ with distribution $\mu$, and an $\cF_0$-measurable $\bR^{d}$-valued random variable $\xi$ with distribution $\nu$ such that
\[
S, \xi, W \text{ are independent}.
\]
Clearly, this last requirement is stronger than the independence of $(S,\xi)$ and $W$ that comes from the $\cF_0$-measurability of the former pair.

Let $\Gamma \subset \bR^{d}$ be a closed set, the space of control actions. Let $\cA$ denote the set of all $\Gamma$-valued $\bF$-progressively measurable processes $\alpha$ such that $\alpha$ is square-integrable in the sense that
\[
\Mean\left[ \int_{0}^{t} |\alpha(r)|^{2}dr \right] < \infty \quad\text{for all }t \geq 0.
\]
Let $T > 0$ be the finite time horizon and $\sigma \geq 0$ the noise intensity parameter. For $\alpha \in \cA$, let $X^{\xi,\alpha}$ denote the process
\begin{equation} \label{EqDynamicsXi}
	X^{\xi,\alpha}(t) \doteq \xi + \int_{0}^{t} \alpha(r)dr + \sigma\cdot W(t), \quad t\geq 0.
\end{equation}
Given a deterministic initial time $t_0 \in [0,T]$ and a deterministic position $x\in \bR^{d}$, let $X^{t_0,x,\alpha}$ denote the process
\begin{equation} \label{EqDynamics}
	X^{t_0,x,\alpha}(t) \doteq x + \int_{t_0}^{t} \alpha(r)dr + \sigma\cdot \left(W(t)-W(t_0) \right), \quad t\geq t_0,
\end{equation}
where we set $X^{t_0,x,\alpha}(t) \doteq x$ for $t\in [0,t_0)$. The processes $X^{\xi,\alpha}$, $X^{t_0,x,\alpha}$ will denote the controlled position process of the representative player when starting from $\xi$ at time zero and from $x$ at time $t_{0}$, respectively.

Let $\bF^{\xi,W}$ denote the filtration generated by $\xi$ and $W$:
\[
\cF^{\xi,W}_t \doteq \boldsymbol{\sigma}\left( \xi, W(r) : r \in [0,t] \right), \quad t \geq 0.
\]
Let $\cA^{\xi}$ denote the set of all $\bF^{\xi,W}$-progressively measurable processes in $\mathcal{A}$, and let $\cT^{\xi}_{T}$ denote the set of all $[0,T]$-valued $\bF^{\xi,W}$-stopping times. For $\tau \in \cT^{\xi}_{T}$, let $\bF^{\xi,W,S(\tau)}$ denote the filtration given by
\[
\cF^{\xi,W,S(\tau)}_t \doteq 
\boldsymbol{\sigma}\left(\xi, W(r), S\cdot \mathbf{1}_{[\tau,\infty)}(r) : r\in [0,t] \right), \quad t\geq 0,
\]
where, recalling the definition of $s_0$, we have for all $\omega \in \Omega$,
\[
S(\omega) \cdot \mathbf{1}_{[\tau(\omega),\infty)}(r) = \begin{cases}
	S(\omega) &\text{if } r \geq \tau(\omega), \\
	s_{0} &\text{if } r < \tau(\omega).
\end{cases}
\]
Denote by $\cA^{\xi}_{\tau}$ the set of all $\bF^{\xi,W,S(\tau)}$-progressively measurable processes in $\mathcal{A}$.

For $t_0 \in [0,T]$, let $\bF^{W,t_0}$ denote the filtration given by
\[
\cF^{W,t_0}_t \doteq \begin{cases}
	\{\emptyset,\Omega\} &\text{if } t \in [0,t_0), \\
	\boldsymbol{\sigma}\left( W(r) - W(t_0) : r\in [t_0,t] \right) &\text{if } t \geq t_0. 
\end{cases}
\]
Let $\cT_{t_0,T}$ denote the set of all $[t_0,T]$-valued $\bF^{W,t_0}$-stopping times. For $\tau \in \cT_{t_0,T}$, let $\bF^{W,t_0,S(\tau)}$ denote the filtration given by
\[
\cF^{W,t_0,S(\tau)}_t \doteq \begin{cases}
	\{\emptyset, \Omega\} &\text{if } t \in [0,t_0), \\
	\boldsymbol{\sigma}\left(W(r) - W(t_0), S\cdot \mathbf{1}_{[\tau,\infty)}(r) : r\in [t_0,t] \right) &\text{if } t \geq t_0.
\end{cases}
\]
Let $\cA_{t_0}$ denote the set of all $\Gamma$-valued $\bF^{W,t_0}$-progressively measurable processes, and let $\cA_{t_0,\tau}$ denote the set of all $\Gamma$-valued $\bF^{W,t_0,S(\tau)}$-progressively measurable processes.

Notice that our control processes are all defined on the entire non-negative half-line, and they are all $\bF$-progressively measurable. By construction, we have, for $0 \leq t_{0} \leq t_{1} \leq T$,
\begin{align*}
	& \cA_{t_1} \subset \cA_{t_0} \subset \cA^{\xi}, & 
	& \cT_{t_1,T} \subset \cT_{t_0,T} \subset \cT^{\xi}_{T}, &
\end{align*}
and, for $\tau \in \cT^{\xi}_{t_0, T}$, $t_0\in[0,T],$
\begin{align*}
	& \cA^{\xi} \subset \cA^{\xi}_{\tau} \subset \cA, &
	& \cA_{t_0} \subset \cA_{t_0,\tau} \subset \cA. &	
\end{align*}

To define the cost functional, choose $f\colon [0,T]\times \bR^{d} \times \bS \times \Gamma \times \cP(\bR^{d}\times\{0,1\}) \rightarrow \bR$, representing the running costs, a function $g\colon \bR^{d} \times \bS \times \cP(\bR^{d}) \rightarrow \bR$, representing the terminal costs, and a function $h\colon [0,T]\times \bR^{d} \rightarrow [0,\infty)$ as the cost of buying access to the hidden state $S$.

\begin{assumptions}
	The functions $f$, $g$, $h$ are Borel measurable and such that:
	
	\begin{hypenv}
		\item \label{ARunTermCosts} The functions $f(t,.,.,.,.)$, $g$ are continuous for all $t\in [0,T]$, and there exists a finite constant $K$ such that for all $x\in \bR^{d}$,
		\[
		\sup_{t\in [0,T], s \in \bS, a \in \Gamma, m \in \cP(\bR^{d}\times\{0,1\})} |f(t,x,s,a,m)|\vee |g(x,s,m)| \leq K\left(1 + |a|^{2} + |x|^{2} \right).
		\]
		
		\item \label{ACoercivity} Either $\Gamma$ is compact or there exist constants $c_{0}$, $C_{0} \in (0,\infty)$, $q\in [0,2)$ such that for all $a\in \Gamma$, all $x\in \bR^{d}$,
		\begin{align*}
			\inf_{t\in [0,T], s \in \bS, m \in \cP(\bR^{d}\times\{0,1\})} f(t,x,s,a,m) & \geq c_{0}|a|^{2} - C_{0}\left(1 + |x|^{q} \right), \\
			\inf_{s \in \bS, m \in \cP(\bR^{d}\times\{0,1\})} g(x,s,m) &\geq - C_{0}\left(1 + |x|^{q} \right).
		\end{align*}
		
		\item \label{AInfoCosts} The function $h$ is bounded non-negative, $h$ is continuous on $[0,T)\times \bR^{d}$ with finite limits from the left at $T$, 
		and $h(T,.) = 0$.
		
		\item \label{ALocLipMeasure} With $\mathrm{d}_{\cP}$ some bounded metric on $\cP(\bR^{d}\times\{0,1\})$ that is compatible with the topology of weak convergence of measures, there exists a finite constant $L$ such that for all $x\in \bR^{d}$, all $m, \tilde{m} \in \cP(\bR^{d}\times\{0,1\})$,
		\begin{multline*}
			\sup_{t\in [0,T], s \in \bS, a \in \Gamma} |f(t,x,s,a,m)-f(t,x,s,a,\tilde{m})|\vee |g(x,s,m)-g(x,s,\tilde{m})| \\
			\leq L\left(1 + |x| \right)\cdot \mathrm{d}_{\cP}(m,\tilde{m}).
		\end{multline*}
		
	\end{hypenv}
\end{assumptions}

In the above assumptions and in the sequel, we identify the terminal cost function $g$, which is defined on $\bR^{d} \times \bS \times \cP(\bR^{d})$, with its natural extension to a function on $\bR^{d} \times \bS \times \cP(\bR^{d}\times \{0,1\})$. The growth assumption \hypref{ARunTermCosts} in conjunction with the coercivity assumption \hypref{ACoercivity} and the boundedness of the information costs according to \hypref{AInfoCosts} will ensure that the expected costs as well as the minimal expected costs, as given below, will always be finite. Assumption~\hypref{ALocLipMeasure} is a condition of Lipschitz continuity in the measure variable that holds locally in the position variable, uniformly in the other variables. A compatible bounded metric $\mathrm{d}_{\cP}$ is given, for example, by the bounded Lipschitz metric; cf.\ Section~11.3 in \citet[pp.\,393\,ff.]{dudley02}.


A mapping $\rho\colon \mathbb{S} \times [0,\! T] \rightarrow \mathcal{P}(\mathbb{R}^{d}\times \{0,\! 1\})$, written $(s,t)\mapsto \rho_{s}(t)$, is called a \emph{conditional flow of measures} if it is Borel measurable. For $(s,t)\in \bS \times [0,T]$, $\rho_{s}(t)$ will be the conditional distribution of the positions and information states of the player population at time $t$ given that the hidden state $S$ equals $s$. In particular, for a Borel set $B \subset \bR^d$, $\rho_{s}(t;B\times\{0\})$ will be the mean field limit proportion of players who at time $t$ stay in $B$ and have not bought information on the hidden state $S$ given that $S = s$, whereas $\rho_{s}(t;B\times\{1\})$ will be the proportion of players who at time $t$ stay in $B$ and have bought information on $S$ given that $S = s$, hence know that $S = s$. 

For $\tau\in \mathcal{T}^{\xi}_{T}$, $\alpha\in \mathcal{A}^{\xi}_{\tau}$, and $\rho$ a conditional flow of measures, define the expected costs under $(\tau,\alpha)$ starting from $\xi$ at time zero according to
\begin{equation*} 
	J_{\rho}(\tau,\alpha) \doteq \Mean\left[ \int_{0}^{T} \! f\left(t,X^{\xi,\alpha}(t),S,\alpha(t),\rho_{S}(t)\right)dt + h\bigl(\tau, X^{\xi,\alpha}(\tau)\bigr) + g\left(X^{\xi,\alpha}(T),S,\rho_{S}(T)\right) \right],
\end{equation*}
where $X^{\xi,\alpha}$ obeys \eqref{EqDynamicsXi}. Then we have the following definition of solutions to our mean field game:

\begin{definition} \label{DefMFGstrong}
	A conditional flow of measures $\rho$ is called a \emph{strong solution of the mean field game} if there exist $\tau\in \mathcal{T}^{\xi}_{T}$ and $\alpha\in \mathcal{A}^{\xi}_{\tau}$ such that
	\begin{itemize}
		
		\item Optimality:
		\[
		J_{\rho}(\tau,\alpha) \leq \inf_{\tilde{\tau}\in \mathcal{T}^{\xi}_{T}} \inf_{\tilde{\alpha}\in \mathcal{A}^{\xi}_{\tilde{\tau}}} J_{\rho}(\tilde{\tau},\tilde{\alpha}).
		\]
		
		\item Consistency: for $\mu$-a.e.\ $s\in \mathbb{S}$, all $t\in [0,\! T]$,
		\[
		\rho_{s}(t) = \Prb\left( (X^{\alpha}(t), \mathbf{1}_{[\tau,\infty)}(t))\in . \mid S = s \right).
		\]
		
	\end{itemize}
	
\end{definition}

Given a conditional flow of measures $\rho$, the minimal costs when starting at time\linebreak[4]$t_0 \in [0,T]$ in position $x\in \bR^{d}$  \emph{not knowing} the value of $S$ are given by
\begin{equation} \label{ExVorg}
	\begin{split}
		V_{\rho}(t_{0}, x) \doteq \inf_{\tau\in \cT_{t_0,T}} \inf_{\alpha \in \cA_{t_0,\tau}} \Mean\Bigl[ &\int_{t_0}^{T}
		f\bigl(t, X^{t_{0},x,\alpha}(t), S, \alpha(t), \rho_{S}(t)\bigr) d t  \\
		&+ h\bigl(\tau, X^{t_{0},x,\alpha}(\tau)\bigr) + g\bigl(X^{t_{0},x,\alpha}(T), S, \rho_{S}(T)\bigr) \Bigr].
	\end{split}
\end{equation}
Notice that the admissible control processes in the definition of $V_{\rho}(t_0,.)$ above are in $\cA_{t_0,\tau}$ and therefore depend on the stopping time $\tau\in \cT_{t_0,T}$. Also note that, thanks to \hyprefallbutlast, $V_{\rho}(t_0,x)\in (-\infty,\infty)$ for all $(t_0,x) \in [0,T]\times \bR^{d}$.

\begin{remark}
	The optimality condition in Definition~\ref{DefMFGstrong} can be rewritten using the value function $V_{\rho}$:
	\[
	\inf_{\tilde{\tau}\in \mathcal{T}^{\xi}_{T}} \inf_{\tilde{\alpha}\in \mathcal{A}^{\xi}_{\tilde{\tau}}} J_{\rho}(\tilde{\tau},\tilde{\alpha}) = \int_{\bR^d} V_{\rho}(0, x) \nu(dx).
	\]
	To check this, one conditions the processes defining $J_{\rho}(\tilde{\tau},\tilde{\alpha})$ on the random variable $\xi$.
\end{remark}


\section{The associated optimal control-stopping problem} \label{SectMixedCP}

Let $\rho$ again be a conditional flow of measures. Given $\rho$ and a hidden state $s \in \bS$, the minimal costs when starting at time $t_0$ in position $x$ \emph{knowing that} $S=s$ are given by
\begin{equation} \label{ExVsknown}
	V_{s,\rho}(t_{0}, x) \doteq \inf _{\alpha \in \cA_{t_0}} \Mean\left[
	\int_{t_{0}}^{T} f\bigl(t, X^{t_0,x,\alpha}(t), s, \alpha(t), \rho_{s}(t) \bigr) d t + g\bigl(X^{t_0,x,\alpha}(T), s, \rho_{s}(T)\bigr)\right]
\end{equation}
where $X^{t_0,x,\alpha}$ with $\alpha \in \cA_{t_0}$ is given by \eqref{EqDynamics}. The control processes are thus adapted to~$\bF^{W,t_0}$, the filtration generated by the increments of the driving Wiener process. Notice that the definition of $ V_{s,\rho}$ does not include the cost of buying the information $S=s$.

\begin{remark}
	Under our assumptions, $V_{s, \rho}$ will coincide with the unique (classical or at least viscosity) solution $v$ to the \mbox{HJB} equation
	\begin{equation} \label{EqHJBVs}
		\begin{cases}
			v(T, x) =  g\bigl(x, s, \rho_{s}(T)\bigr), & x \in \bR^{d}, \\
			-\frac{\partial}{\partial t} v(t, x) = H_{s,\rho}\bigl(t,x, \nabla_{x} v(t, x) \bigr) + \frac{\sigma^{2}}{2} \Delta_{x} v(t, x),
			&(t, x) \in[0, T) \times \bR^{d},
		\end{cases}
	\end{equation}
	where, with $\langle .\,,.\rangle$ denoting the standard scalar product,
	\[
	H_{s,\rho}(t,x,p) \doteq \inf_{a \in \Gamma}
	\left\{ f\left(t, x, s, a, \rho_{s}(t)\right) + \left\langle a,p \right\rangle \right\}.
	\]
\end{remark}

Recall that $\mu$ is the law of $S$. Define $\tilde{f}_{\rho} \colon [0,T] \times \bR^{d} \times \Gamma \rightarrow \bR$, $\tilde{V}_{\rho} \colon [0,T] \times \bR^{d} \rightarrow \bR$, and~$G_{\rho}\colon [0,T] \times \bR^{d} \rightarrow \bR$ according to
\begin{align*}
	& \tilde{f}_{\rho}(t, x, a) \doteq \int_{\bS} f\bigl(t, x, s, a, \rho_{s}(t)\bigr) \mu(ds),& & \tilde{V}_{\rho}(t,x) \doteq \int_{\bS} V_{s,\rho}(t, x)\, \mu(ds),& \\
	& G_{\rho}(t, x) \doteq h(t,x) + \tilde{V}_{\rho}(t,x). & &&
\end{align*}
Let $\hat{V}_{\rho} \colon [0,T]\times \bR^{d} \rightarrow \bR$ be given by
\begin{equation} \label{ExVhat}
	\hat{V}_{\rho}(t_{0}, x) \doteq \inf_{\tau \in \cT_{t_0,T}} \inf_{\alpha \in \cA_{t_0}} \Mean\left[ \int_{t_0}^{\tau} \tilde{f}_{\rho}\bigl(t, X^{t_0,x,\alpha}(t), \alpha(t)\bigr) d t + G_{\rho}\bigl(\tau, X^{t_0,x,\alpha}(\tau)\bigr) \right].
\end{equation}
Then $\hat{V}_{\rho}$ is the value function of an optimal control problem with discretionary stopping. Recall that the admissible stopping times for $\hat{V}_{\rho}(t_0,.)$ are $[t_0,T]$-valued. Also note that both the admissible stopping times and the admissible control processes are adapted to~$\bF^{W,t_0}$.

\begin{remark}
	The value $\hat{V}_{\rho}$ should coincide with the unique viscosity solution $\hat{v}$ to
	\begin{equation} \label{EqHJBVhat}
		\begin{cases}
			\hat{v}(T, x) = G_{\rho}(T,x), & x \in \bR^{d}, \\
			\min \Bigl\{ G_{\rho}(t,x) - \hat{v}(t,x), & \\
			\qquad \frac{\partial}{\partial t} \hat{v}(t, x) + H_{\rho}\bigl(t, x, \nabla_{x} \hat{v}(t,x) \bigr) + \frac{\sigma^{2}}{2} \Delta_{x} \hat{v}(t, x) \Bigr\} = 0, & (t, x) \in[0, T) \times \bR^{d},
		\end{cases}
	\end{equation}
	where
	\[
	H_{\rho}(t,x,p) \doteq \inf_{a \in \Gamma}\left\{ \tilde{f}_{\rho}(t, x, a) + \left\langle a, p \right\rangle \right\},
	\]
	see Section~4 in \citet{elasrietal22} (there, the objective functional is maximized, not minimized as here).
\end{remark}

The original control problem with option to buy information and the associated optimal control problem with discretionary stopping are equivalent in the following sense:

\begin{thrm} \label{ThValueEquiv}
	Grant \hyprefallbutlast. Let $\rho$ be a conditional flow of measures. Let $V_{\rho}$ be given by \eqref{ExVorg}, and let $\hat{V}_{\rho}$ be given by \eqref{ExVhat}. Then $V_{\rho} = \hat{V}_{\rho}$.
\end{thrm}

\begin{proof}
	Let $x \in \bR^{d}$. By construction, since $\cT_{T,T}$ only contains the constant stopping time equal to $T$ and $h(T,.) = 0$ by assumption, we have
	\begin{multline*}
		V_{\rho}(T,x) = \Mean\left[ g\bigl(X^{T,x,\alpha}(T), S, \rho_{S}(T)\bigr) \right] = \Mean\left[ g\bigl(x, S, \rho_{S}(T)\bigr) \right] \\
		= \int_{\bS} g\bigl(x, s, \rho_{s}(T)\bigr) \mu(d s) = \int_{\bS} V_{s,\rho}(T,x) \mu(d s) = G_{\rho}(T,x) = \hat{V}_{\rho}(T,x).
	\end{multline*}
	
	Now, let $t_0 \in [0,T)$, and let $\epsilon > 0$ be arbitrary. We are going to show first that $\epsilon + V_{\rho}(t_0,x) \geq \hat{V}_{\rho}(t_0,x)$, then that $2\epsilon + \hat{V}_{\rho}(t_0,x) \geq V_{\rho}(t_0,x)$. This will imply the assertion.
	
	\paragraph{First step: $\epsilon + V_{\rho}(t_0,x) \geq \hat{V}_{\rho}(t_0,x)$.} Let $(\tau,\alpha) \in \cT_{t_0,T}\times \cA_{t_0,\tau}$ be $\epsilon$-optimal for the original control problem starting in $(t_0,x)$, that is, $(\tau,\alpha)$ is such that
	\[
	\epsilon + V_{\rho}(t_{0}, x) \geq \Mean\left[ \int_{t_0}^{T}
	f\bigl(t, X(t), S, \alpha(t), \rho_{S}(t)\bigr) d t  + h\bigl(\tau, X(\tau)\bigr) + g\bigl(X(T), S, \rho_{S}(T)\bigr) \right],
	\]
	where $X \doteq X^{t_{0},x,\alpha}$.
	
	Fix $\gamma_0 \in \Gamma$, and set for $t\geq 0$, $\omega\in \Omega$,
	\[
	\hat{\alpha}(t,\omega) \doteq \begin{cases}
		\alpha(t,\omega) &\text{if } t < \tau(\omega), \\
		\gamma_0 &\text{if } t \geq \tau(\omega).
	\end{cases}
	\]
	Then $\hat{\alpha}$ is $\bF^{W,t_0}$-progressively measurable, hence in $\cA_{t_0}$. To check this, let $t \geq 0$. Recall that $\tau$ takes values in $[t_0,T]$. If $t\in [0,t_0)$, then $\hat{\alpha}(t) = \alpha(t)$ and $\alpha(t)$ is deterministic. Now, suppose $t \geq t_0$. Since $\alpha(t)$ is $\cF^{W,t_0,S(\tau)}_t$-measurable and $\Gamma$ is a Borel space, by Doob's functional representation \citep[for instance, Lemma~1.14 in][p.\,18]{kallenberg21}, there exists a function $\psi\colon (\bR^{d} \times \bS)^{[t_0,t]} \rightarrow \Gamma$ measurable with respect to the infinite product $\sigma$-algebra of the Borel $\sigma$-algebra in $\bR^{d} \times \bS$ and to the Borel $\sigma$-algebra in $\Gamma$ such that for all $\omega\in \Omega$,
	\[
	\alpha(t,\omega) = \psi\left( \left(W(r,\omega)-W(t_0,\omega), S(\omega)\cdot \mathbf{1}_{[\tau(\omega),\infty)}(r) \right)_{r\in [t_0,t]} \right),
	\]
	hence
	\[
	\hat{\alpha}(t,\omega) = \begin{cases}
		\psi\left( \bigl(W(r,\omega)-W(t_0,\omega), s_0 \bigr)_{r\in [t_0,t]} \right) &\text{if } t\in [t_0,\tau(\omega)), \\
		\gamma_{0} &\text{if } t\geq \tau(\omega).
	\end{cases}
	\]
	But $\tau$ is an $\bF^{W,t_0}$-stopping time, which implies that the event $\{t < \tau \}$ is in $\cF^{W,t_0}_t$. It follows that $\hat{\alpha}(t)$ is $\cF^{W,t_0}_t$-measurable.
	
	Denote by $\hat{X}$ the process $X^{t_{0},x,\hat{\alpha}}$. By construction of $\hat{\alpha}$ and \eqref{EqDynamics}, we have for all~$\omega\in \Omega$,
	\begin{equation} \label{EqValueEquivXhat}
		X(t,\omega) = \hat{X}(t,\omega) \text{ whenever } t\in [t_{0}, \tau(\omega)).
	\end{equation}
	By continuity of trajectories, \eqref{EqValueEquivXhat} implies that
	\[
	X(\tau) = \hat{X}(\tau)\quad \Prb\text{-almost surely}.
	\]
	The fact that $\hat{\alpha}$ is $\bF^{W,t_0}$-progressively measurable entails that $\hat{X}$ is $\bF^{W,t_0}$-progressively measurable, too. Since $S$ and $W$ are independent, we see that $S$ and $(\tau,\hat{X},\hat{\alpha})$ are independent.
	Therefore, by \eqref{EqValueEquivXhat}, independence and Fubini's theorem,
	\begin{align*}
		& \Mean\left[ \int_{t_0}^{\tau} f\bigl(t, X(t), S, \alpha(t), \rho_{S}(t)\bigr) d t  + h\bigl(\tau, X(\tau)\bigr) \right] \\
		&= \Mean\left[ \int_{\bS} \int_{t_0}^{T} \mathbf{1}_{[0,\tau)}(t)\cdot f\bigl(t, \hat{X}(t), s, \hat{\alpha}(t), \rho_{s}(t)\bigr) d t\, \mu(ds)  + h\bigl(\tau, \hat{X}(\tau)\bigr) \right] \\
		&= \Mean\left[ \int_{t_0}^{\tau} \tilde{f}_{\rho} \bigl(t, \hat{X}(t), s, \hat{\alpha}(t), \rho_{s}(t)\bigr) \mu(ds)\, d t  + h\bigl(\tau, \hat{X}(\tau)\bigr) \right].
	\end{align*}
	
	It thus remains to show:
	\begin{equation} \label{EqValueEquivStepOne}
		\Mean\left[ \int_{\tau}^{T} f\bigl(t, X(t), S, \alpha(t), \rho_{S}(t)\bigr) d t  + g\bigl(X(T), S, \rho_{S}(T)\bigr) \right] \geq \Mean\left[ \tilde{V}_{\rho}\bigl(\tau, X(\tau)\bigr) \right].
	\end{equation}
	Since $S$ and $(\tau, \hat{X}(\tau))$ are independent under $\Prb$, we also have that $S$ and $(\tau,X(\tau))$ are independent. Using Fubini's theorem, it follows that
	\[
	\Mean\left[ \tilde{V}_{\rho}\bigl(\tau, X(\tau)\bigr) \right] = \int_{\bS} \Mean\left[ V_{s,\rho}\bigl(\tau, X(\tau)\bigr) \right] \mu(ds) = \int_{\bS\times [t_0,T]\times \bR^{d}} V_{s,\rho}(\mathfrak{t},y)\, Q(ds,d\mathfrak{t},dy),
	\]
	where $Q \doteq \Prb\circ (S,\tau,X(\tau))^{-1}$ denotes the joint distribution of $S$, $\tau$, $X(\tau)$ and, by the above independence, $Q = \mu\otimes Q_{2}$ with $Q_{2} \doteq \Prb\circ (\tau,X(\tau))^{-1}$.
	
	Let $\kappa$ denote (a version of) the regular conditional distribution under $\Prb$ of $(\alpha,X,W)$ given $(S,\tau,X(\tau))$. Thus, for $(s,\mathfrak{t},y) \in \bS\times [t_0,T]\times \bR^{d}$, $\kappa_{(s,\mathfrak{t}, y)}$ denotes the conditional distribution of $(\alpha,X,W)$ given that $S = s$, $\tau = \mathfrak{t}$, and $X(\tau) = y$. 
	Here, ordinary control processes are interpreted as relaxed controls (for instance, Section~2.1 in \citet{lacker17}) so that we can see them as random elements in a Polish space. The existence of a regular conditional distribution $\kappa$ is then guaranteed, and $\kappa$ is uniquely determined up to sets of $Q$-measure zero; cf.\ Theorem~8.5 in \citet[p.\,168]{kallenberg21}.
	
	Recall that $W$ is an $\bF$-Wiener process, $S$ is $\cF_{0}$-measurable, that $\bF^{W,t_0} \subseteq \bF^{W,t_0,S(\tau)} \subseteq \bF$, that $\tau$ is an $\bF^{W,t_0}$-stopping time and that $\alpha$ is $\bF^{W,t_0,S(\tau)}$-progressively measurable. This implies that the process $(W(\tau+r)-W(\tau))_{r\geq0}$ is a Wiener process with respect to the filtration $(\cF_{\tau+r})_{r\geq 0}$. Moreover, $(S,\tau,X(\tau))$ and $(W(\tau+r)-W(\tau))_{r\geq0}$ are independent under $\Prb$. In view of \eqref{EqDynamics}, we have
	\[
	X(t) = X(\tau) + \int_{\tau}^{t} \alpha(r)dr + \sigma\cdot \left( W(t)-W(\tau) \right) \text{ whenever } t\geq \tau.
	\]
	Notice that the process $(\alpha(\tau + r))_{r\geq 0}$ is $(\cF_{\tau+r})_{r\geq 0}$-progressively measurable. It follows from the above that for $Q$-almost every $(s,\mathfrak{t},y) \in \bS\times [t_0,T]\times \bR^{d}$, we have with $\kappa_{(s,\mathfrak{t}, y)}$-probability one:
	\begin{equation} \label{EqValueEquivCondDyn}
		X(t) = y + \int_{\mathfrak{t}}^{t} \alpha(r)dr + \sigma\cdot \left( W(t)-W(\mathfrak{t}) \right) \text{ for all } t\geq \mathfrak{t},
	\end{equation}
	where $(W(\mathfrak{t}+r)-W(\mathfrak{t}))_{r\geq 0}$ is still a Wiener process and, for every $t \geq \mathfrak{t}$, $(X(r), \alpha(r))_{r \in [\mathfrak{t},t]}$ and $(W(r)- W(t))_{r \geq t}$ are independent. Strictly speaking, we should see the processes $X$, $\alpha$, $W$ under $\kappa_{(s,\mathfrak{t}, y)}$ as projections on a suitable canonical space. With slight abuse of notation, write $\Mean_{(s,\mathfrak{t}, y)}$ for the expectation with respect to $\kappa_{(s,\mathfrak{t}, y)}$. Then
	\begin{multline*}
		\Mean\left[ \int_{\tau}^{T} f\bigl(t, X(t), S, \alpha(t), \rho_{S}(t)\bigr) d t  + g\bigl(X(T), S, \rho_{S}(T)\bigr) \right] \\
		= \int_{\bS\times [t_0,T]\times \bR^{d}} \Mean_{(s,\mathfrak{t}, y)} \left[ \int_{\mathfrak{t}}^{T} f\bigl(t, X(t), s, \alpha(t), \rho_{s}(t)\bigr) d t  + g\bigl(X(T), S, \rho_{s}(T)\bigr) \right] Q(ds, d\mathfrak{t}, y).
	\end{multline*}
	Recall that, under $\kappa_{(s,\mathfrak{t}, y)}$, $X$ solves \eqref{EqValueEquivCondDyn} and the control process $\alpha$ is non-anticipative with respect to the driving Wiener process. But $\alpha$ is not necessarily adapted to the filtration generated by the driving Wiener process, as instead required by the definition of~$V_{s,\rho}$ in \eqref{ExVsknown}, where the value function $V_{s,\rho}$ is given in the so-called strong formulation. However, the value function would not change if $V_{s,\rho}$ had been defined in the weak formulation, that is, if the infimum in \eqref{ExVsknown} had been taken over all probability spaces carrying a filtration $\bG$, a $d$-dimensional Wiener process with respect to $\bG$, and a $\bG$-progressively measurable $\Gamma$-valued control process. This equivalence of weak and strong formulation for $V_{s,\rho}$ follows, for instance, from Theorem~2.4 in \citet{lacker17}, where finite horizon stochastic optimal control problems of McKean-Vlasov type are treated. Consequently, we find that for $Q$-almost every $(s,\mathfrak{t},y) \in \bS\times [t_0,T]\times \bR^{d}$,
	\[
	\Mean_{(s,\mathfrak{t}, y)} \left[ \int_{\mathfrak{t}}^{T} f\bigl(t, X(t), s, \alpha(t), \rho_{s}(t)\bigr) d t  + g\bigl(X(T), S, \rho_{s}(T)\bigr) \right] \geq V_{s,\rho}(\mathfrak{t},y),
	\]
	hence
	\[
	\Mean\left[ \int_{\tau}^{T} f\bigl(t, X(t), S, \alpha(t), \rho_{S}(t)\bigr) d t  + g\bigl(X(T), S, \rho_{S}(T)\bigr) \right] \\ \geq \int_{\bS\times [t_0,T]\times \bR^{d}} V_{s,\rho}(\mathfrak{t},y)\, Q(ds, d\mathfrak{t}, dy),
	\]
	which establishes \eqref{EqValueEquivStepOne}.
	
	\paragraph{Second step: $2\epsilon + \hat{V}_{\rho}(t_0,x) \geq V_{\rho}(t_0,x)$.} Let $(\tau,\tilde{\alpha}) \in \cT_{t_0,T}\times \cA_{t_0}$ be $\epsilon$-optimal for the associated control problem starting in $(t_0,x)$, that is, $(\tau,\tilde{\alpha})$ is such that
	\[
	\epsilon + \hat{V}_{\rho}(t_{0}, x) \geq \Mean\left[ \int_{t_0}^{\tau} \tilde{f}_{\rho}\bigl(t, \tilde{X}(t), \tilde{\alpha}(t)\bigr) d t  + h\bigl(\tau, \tilde{X}(\tau)\bigr) + \tilde{V}_{\rho}\bigl(\tau, \tilde{X}(\tau)\bigr) \right],
	\]
	where $\tilde{X} \doteq X^{t_{0},x,\tilde{\alpha}}$. Notice that $\tilde{\alpha}$ is $\bF^{W,t_0}$-progressively measurable. Hence the same is true for $\tilde{X}$, and since $\tau$ is an $\bF^{W,t_0}$-stopping time, we find that $S$ and $(\tau,\tilde{X},\tilde{\alpha})$ are independent. It follows that
	\begin{align*}
		\Mean\left[ \int_{t_0}^{\tau} \tilde{f}_{\rho}\bigl(t, \tilde{X}(t), \tilde{\alpha}(t)\bigr) d t \right] &= \Mean\left[ \int_{t_0}^{\tau} f\bigl(t, \tilde{X}(t), S, \tilde{\alpha}(t), \rho_{S}(t)\bigr) d t \right].
	\end{align*}
	
	It remains to show that there exists $\alpha \in \cA_{t_0,\tau}$ such that
	\begin{subequations}
		\begin{align}
			\label{EqValueEquivAlphaTau}
			&\alpha(t,\omega) \cdot \mathbf{1}_{[0,\tau(\omega))}(t) = \tilde{\alpha}(t,\omega) \cdot \mathbf{1}_{[0,\tau(\omega))}(t) \text{ for } \lambda_1\otimes \Prb\text{-a.a.\ } (t,\omega) \in [0,\infty)\times \Omega, \\
			\intertext{and}
			\label{EqValueEquivAlphaOpt}
			&\epsilon + \Mean\left[ \tilde{V}_{\rho}\bigl(\tau, \tilde{X}(\tau)\bigr) \right] \geq \Mean\left[ \int_{\tau}^{T} f\bigl(t, X(t), S, \alpha(t), \rho_{S}(t)\bigr) d t  + g\bigl(X(T), S, \rho_{S}(T)\bigr) \right],
		\end{align}
	\end{subequations}
	where $X \doteq X^{t_{0},x,\alpha}$ and $\lambda_1$ denotes one-dimensional Lebesgue measure.
	
	By the independence of $S$ and $(\tau, \tilde{X}(\tau))$, we have
	\[
	\Mean\left[ \tilde{V}_{\rho}\bigl(\tau, \tilde{X}(\tau)\bigr) \right] = \Mean\left[ V_{S,\rho}\bigl(\tau, \tilde{X}(\tau)\bigr) \right].
	\]
	If $\alpha \in \cA_{t_0,\tau}$ is such that $\eqref{EqValueEquivAlphaTau}$ holds, then by continuity of trajectories,
	\[
	X(\tau) = \tilde{X}(\tau)	\quad \Prb\text{-almost surely}.
	\]
	In order to construct $\alpha \in \cA_{t_0,\tau}$ as desired, it is enough to verify that there exists a function $\psi\colon \bS\times [t_0,T]\times \bR^{d} \rightarrow \mathcal{A}$ such that, writing $\psi$ as the mapping $(s,\mathfrak{t},y) \mapsto \psi_{(s,\mathfrak{t},y)}$, we have:
	\begin{enumerate}[(i)]
		\item for every $(s,\mathfrak{t},y) \in \bS\times [t_0,T]\times \bR^{d}$, the process $\psi_{(s,\mathfrak{t},y)}$ is $\bF^{W,\mathfrak{t}}$-progressively measurable, hence in $\mathcal{A}_{\mathfrak{t}}$, and it is $\epsilon$-optimal for $V_{s,\rho}(\mathfrak{t},y)$, that is,
		\[
		\epsilon + V_{s,\rho}(\mathfrak{t},y) \geq \Mean\left[
		\int_{\mathfrak{t}}^{T} f\bigl(t, X^{\mathfrak{t},y,\psi_{(s,\mathfrak{t},y)}}(t), s, \alpha(t), \rho_{s}(t) \bigr) d t + g\bigl(X^{\mathfrak{t},y,\psi_{(s,\mathfrak{t},y)}}(T), s, \rho_{s}(T)\bigr)\right],
		\]
		where $X^{\mathfrak{t},y,\psi_{(s,\mathfrak{t},y)}}$ is given by \eqref{EqDynamics};
		
		\item for every $t \geq 0$, the mapping
		\[
		\bS\times [t_0,T]\times \bR^{d} \times [0,t]\times \Omega \ni (s,\mathfrak{t},y,r,\omega) \mapsto \psi_{(s,\mathfrak{t},y)}(r,\omega) \in \Gamma
		\]
		is $\cB(\bS)\otimes \cB([t_0,T]) \otimes \cB(\bR^{d}) \otimes \cB([0,t]) \otimes \cF^{W,t_{0}}$\,/\,$\cB(\Gamma)$-measurable.
	\end{enumerate}
	Suppose that we have such a function $\psi$. Then we can define $\alpha \in \cA_{t_0,\tau}$ by
	\[
	\alpha(t,\omega) \doteq \begin{cases}
		\tilde{\alpha}(t,\omega) &\text{if } t < \tau(\omega), \\
		\psi_{(S(\omega),\tau(\omega),\tilde{X}(\tau(\omega),\omega))}(t,\omega) &\text{if } t \geq \tau(\omega),
	\end{cases}
	\qquad (t,\omega)\in [0,\infty)\times \Omega.
	\]
	This process $\alpha$, by construction and the properties of $\psi$, will be $\bF^{W,t_0,S(\tau)}$-progressively measurable and will satisfy $\eqref{EqValueEquivAlphaTau}$ as well as the $\epsilon$-optimality condition $\eqref{EqValueEquivAlphaOpt}$.
	
	Existence of $\psi$ with the required measurability conditions can be obtained as follows (recall that $\bS$ is finite): For each $s\in \bS$, choose a countable subset $D_{s} \subset [t_0,T]\times \bR^d$ of isolated points such that if $(\mathfrak{t},y)\in D_{s}$ and $\alpha \in \cA_{\mathfrak{t}}$ is $\epsilon/2$-optimal for $V_{s,\rho}(\mathfrak{t},y)$, then $\alpha$ is $\epsilon$-optimal for $V_{s,\rho}(\tilde{t},\tilde{y})$ whenever $(\tilde{t},\tilde{y})\in [t_0,\mathfrak{t}]\times \bR^d$ and $(\mathfrak{t},y)$ is a nearest neighbor of $(\tilde{t},\tilde{y})$ with respect to $D_{s}$, where only grid points with time component equal to or greater than $\tilde{t}$ are considered. Such a choice is possible thanks to assumptions \hypref{ARunTermCosts} and \hypref{ACoercivity} on $f$ and~$g$. Now, for each $s\in \bS$, each $(\mathfrak{t},y)\in D_{s}$ choose $\alpha_{s,\mathfrak{t},y}$ arbitrarily among all strategies in~$\cA_{\mathfrak{t}}$ that are $\epsilon/2$-optimal optimal for $V_{s,\rho}(\mathfrak{t},y)$, and set $\psi_{(s,\mathfrak{t},y)} \doteq \alpha_{s,\mathfrak{t},y}$. In order to extend~$\psi$ to the entire domain $\bS\times [t_0,T]\times \bR^{d}$, fix any deterministic procedure that produces a measurable nearest neighbor partition of $[t_0,T]\times \bR^{d}$ given a countable set of isolated points $D_{s}$, respecting the constraint on the time component. For $(s,\tilde{t},\tilde{y}) \in \bS\times [t_0,T]\times \bR^{d}$, define $\psi_{(s,\tilde{t},\tilde{y})} \doteq \alpha_{s,\mathfrak{t},y}$ where $(\mathfrak{t},y) \in D_s$ is the nearest neighbor (according to the fixed procedure) of $(\tilde{t},\tilde{y})$. Then $\psi$ enjoys the required properties.
	
	
\end{proof}

\section{Derivation of the mean field game system} \label{SectMFGSystem}

In this section, we derive a coupled system of Hamilton-Jacobi-Bellman equations and a Kolmogorov forward equation that, in principle, can be used to characterize solutions of our mean field game. Recall that $\bS$ is a finite set. For simplicity, suppose that $\mu(\{s\}) > 0$ for all $s\in \bS$, that $\sigma > 0$, and also that $\Gamma$ is compact.

Suppose that our probability space $(\Omega,\cF, \Prb)$ with filtration $\bF$ carries, in addition to the $\bF$-Wiener process $W$ and the $\cF_0$-measurable random variable $S$, an $\bR^d$-valued $\cF_0$-measurable random variable $\xi$ such that $\Prb\circ\, \xi^{-1} = \nu$ and such that $\xi$, $S$, $W$ are independent.

Suppose that the position dynamics are governed by Markov feedback strategies in the following way: Let $\hat{u}\colon [0,T]\times \bR^{d} \rightarrow \Gamma$ and $u\colon \bS\times [0,T]\times \bR^{d} \rightarrow \Gamma$ be measurable functions. Write the mapping $u$ as $(s,t,x)  \mapsto u_{s}(t,x)$. We set $\hat{u}(t,x) = u_{s}(t,x) \doteq 0$ whenever $t > T$.

Let $\hat{X}$ be the (path-wise) unique solution to
\begin{equation} \label{EqMarkovSDEhat}
	\hat{X}(t) = \xi + \int_{0}^{t} \hat{u}\bigl(r,\hat{X}(r)\bigr)dr + \sigma W(t),\quad t\geq 0.
\end{equation}
Since $\Gamma$ is bounded and $\sigma > 0$ by assumption, we see that equation~\eqref{EqMarkovSDEhat} possesses a unique strong solution; see Theorem~2.8 in \citet{gyongykrylov96}. In particular, $\hat{X}$ is adapted to the filtration generated by $\xi$ and $W$.

Let $\Psi\colon [0,T]\times \bR^{d} \rightarrow \mathbb{R}$ be a function of at most polynomial growth (in the space variable, uniformly in time), continuous on $[0,T)\times \bR^{d}$, with $\Psi(T,.) = 0$. Set
\[
D\doteq \left\{ (t,x) \in [0,T]\times \bR^{d} : \Psi(t,x) > 0 \right\}.
\]
By continuity of $\Psi$, the set $D$ is open in $[0,T]\times \bR^{d}$. Assume that $\Psi$ is in $\mathbf{C}^{1,2}(D)$ with continuous partial derivatives on $D$ that are of at most polynomial growth. In addition, suppose that $\Psi$ is such that $D$ has a Lipschitz boundary. Set
\begin{align*}
	& \tau(\omega) \doteq \inf\left\{ t\in [0,T] : (t,\hat{X}(t,\omega)) \notin D \right\}, & \bar{\tau}(\omega) \doteq \tau(\omega) \wedge T,& &\omega\in \Omega.&
\end{align*}
Notice that $\tau$, $\bar{\tau}$ take values in $[0,T] \cup \{\infty\}$ and $[0,T]$, respectively, and that they are stopping times with respect to both $\bF$ and the filtration generated by $\xi$ and $W(.)$. 

Let $s \in \bS$. As $\xi$, $W$, $S$ are independent under $\Prb$, $W$ is a Wiener process also under the conditional probability $\Prb_{s}\doteq \Prb(\,.\mid S=s)$. Moreover, $(\xi,\hat{X},W,\bar{\tau})$ has the same distribution under $\Prb_{s}$ as under $\Prb$ and, by the strong Markov property, $W(\bar{\tau}+.)-W(\bar{\tau})$ is a Wiener process under $\Prb$ as well as under $\Prb_{s}$.

For $s \in \bS$, let $\bar{X}_{s}$ be the (path-wise) unique solution to
\begin{equation} \label{EqMarkovSDEknown}
	\bar{X}_{s}(t) = \hat{X}(\bar{\tau}) + \int_{0}^{t} u_{s}\bigl(\bar{\tau} + r,\bar{X}_{s}(r)\bigr)dr + \sigma \left(W(\bar{\tau}+t) - W(\bar{\tau}) \right),\quad t\geq 0.
\end{equation}
Since $W(\bar{\tau}+.)-W(\bar{\tau})$ and $(\bar{\tau}, \hat{X}(\bar{\tau}))$ are independent under $\Prb$ as well as under $\Prb_{s}$, we can again invoke (the proof of) Theorem~2.8 in \citet{gyongykrylov96}. This yields the existence of the path-wise unique solution $\bar{X}_{s}$ to equation \eqref{EqMarkovSDEknown}. Moreover, $\bar{X}_{s}$ is adapted to the filtration generated by $\hat{X}(\bar{\tau})$ and $W(\bar{\tau}+.)-W(\bar{\tau})$.

Define the $\bR^{d}$-valued position process $X$ according to
\begin{equation*}
	X(t) \doteq
	\begin{cases}
		\hat{X}(t) &\text{if } t < \tau, \\
		\bar{X}_{s}(t-\tau) &\text{if } t \geq \tau \text{ and } S = s,
	\end{cases}
	\quad t\in [0,T],
\end{equation*}
and define the non-negative process $Z$ by
\begin{equation*}
	Z(t) \doteq \begin{cases}
		\Psi\bigl(t,\hat{X}(t)\bigr) &\text{if } t < \tau,\\
		0 &\text{if } t \geq \tau,
	\end{cases}
	\quad t\in [0,T].
\end{equation*}
Then, by definition of $\tau$ and continuity of $\Psi$ (before $T$) and $\hat{X}$, we have for $\Prb$-almost all as well as $\Prb_{s}$-almost all $\omega \in \Omega$ that for all $t\in [0,T)$,
\begin{align} \label{EqMarkovZtau}
	& t < \tau(\omega) \text{ if and only if } Z(t,\omega) > 0, & & t \geq \tau(\omega) \text{ if and only if } Z(t,\omega) = 0. &
\end{align}
Define $b\colon \bS \times [0,T]\times \bR^{d}\times \mathbb{R} \rightarrow \Gamma$ by
\[
b_{s}(t,x,z) \doteq \begin{cases}
	\hat{u}(t,x) &\text{if } z > 0,\\
	u_{s}(t,x) &\text{if } z \leq 0.
\end{cases}
\]
The position process $X$ then solves $\Prb$-almost surely,
\begin{equation} \label{EqMarkovSDE}
	X(t) = \xi + \int_{0}^{t} b_{S}\bigl(r,X(r),Z(r)\bigr)dr + \sigma W(t),\quad t\in [0,T],
\end{equation}
and $\Prb_{s}$-almost surely,
\begin{equation} \label{EqMarkovSDEcond}
	X(t) = \xi + \int_{0}^{t} b_{s}\bigl(r,X(r),Z(r)\bigr)dr + \sigma W(t),\quad t\in [0,T].
\end{equation}

Let $\phi \in \mathbf{C}^{2}_{c}(\bR^{d}\times \bR)$ be a twice continuously differentiable test function with compact support. Then, in view of \eqref{EqMarkovZtau}, for all $t\in [0,T]$,
\begin{equation} \label{EqMarkovTestfnct}
	\phi\bigl( X(t), Z(t) \bigr) = \phi\bigl( X(t\wedge\tau), Z(t\wedge\tau) \bigr) - \phi\bigl( X(t\wedge\tau), 0 \bigr) + \phi\bigl( X(t), 0 \bigr).
\end{equation}

Let $\Mean_{s}$ denote expectation with respect to $\Prb_{s}$. 
Applying It\^{o}'s formula to the process~$t\mapsto\phi(X(t),0)$, followed by Fubini's theorem, we obtain, for all $t\in [0,T]$,
\begin{multline*}
	\Mean_{s}\left[ \phi\bigl( X(t),0 \bigr)\right] = \Mean_{s}\left[ \phi(\xi,0)\right] \\
	+ \int_{0}^{t} \Mean_{s}\left[ \left\langle b_{s}\bigl(r,X(r),Z(r)\bigr), \nabla_{x} \phi\bigl( X(r),0 \bigr) \right\rangle + \frac{\sigma^2}{2} \Delta_{x} \phi\bigl( X(r),0 \bigr) \right]dr.
\end{multline*}
Similarly, recalling \eqref{EqMarkovZtau},
\begin{multline*}
	\Mean_{s}\left[ \phi\bigl( X(t\wedge\tau),0 \bigr)\right] = \Mean_{s}\left[ \phi(\xi,0)\right] \\
	+  \int_{0}^{t} \Mean_{s}\left[ \mathbf{1}_{(0,\infty)}(Z(r))\cdot \left(  \left\langle b_{s}\bigl(r,X(r),Z(r)\bigr), \nabla_{x} \phi\bigl( X(r),0 \bigr) \right\rangle + \frac{\sigma^2}{2} \Delta_{x} \phi\bigl( X(r),0 \bigr) \right)\right] dr.
\end{multline*}
Moreover, recalling that $Z(t)=\Psi(t,X(t))$ for $t<\tau$, an application of It\^{o}'s formula yields
\begin{multline*}
	dZ(t)
	= \left\langle b_s(t,X(t),Z(t)),\nabla_x\Psi(t,X(t))\right\rangle  dt + \left(\frac{\partial\Psi}{\partial t} +  \frac{\sigma^2}{2}\Delta_x\Psi\right)(t,X(t)) dt\\
	+\sigma\left\langle\nabla_x\Psi(t,X(t)),dW(t)\right\rangle, \qquad t<\tau.
\end{multline*}
Consequently, applying It\^{o}'s  formula to $\phi(X(t\wedge\tau),Z(t\wedge\tau))$ and Fubini's theorem, 
and using the regularity assumptions on $\Psi$ and the boundary of $D$, we obtain, for every $t\in[0,T)$,
\begin{align*}
	\Mean_{s} &\left[ \phi\bigl( X(t\wedge\tau), Z(t\wedge\tau) \bigr)\right] = \Mean_{s}\left[ \phi\bigl(\xi, \Psi(0,\xi)\vee 0 \bigr)\right] \\
	&+ \int_{0}^{t} \Mean_{s}\left[ \mathbf{1}_{(0,\infty)}(Z(r))\cdot \Bigl(  \left\langle b_{s}\bigl(r,X(r),Z(r)\bigr), \nabla_{x} \phi\bigl(X(r),Z(r)\bigr) \right\rangle \Bigr) dr \right] \\
	&+ \int_{0}^{t} \Mean_{s}\left[ \mathbf{1}_{(0,\infty)}(Z(r))\cdot \frac{\sigma^2}{2} \Delta_{x} \phi\bigl(X(r),Z(r)\bigr) \Bigr) dr \right] \\
	&+ \int_{0}^{t} \Mean_{s}\left[ \mathbf{1}_{(0,\infty)}(Z(r))\cdot \frac{\partial \phi}{\partial z}\bigl(X(r),Z(r)\bigr) \cdot  \left\langle b_{s}\bigl(r,X(r),Z(r)\bigr), \nabla_{x} \Psi\bigl(r,X(r)\bigr) \right\rangle  dr \right] \\
	&+ \int_{0}^{t} \Mean_{s}\left[ \mathbf{1}_{(0,\infty)}(Z(r))\cdot \frac{\partial \phi}{\partial z}\bigl(X(r),Z(r)\bigr) \cdot  \Bigl( \frac{\partial}{\partial r}\Psi + \frac{\sigma^2}{2} \Delta_{x} \Psi \Bigr) \bigl(r,X(r)\bigr) dr \right] \\
	&+ \int_{0}^{t} \Mean_{s}\left[ \mathbf{1}_{(0,\infty)}(Z(r))\cdot \frac{\partial^{2} \phi}{\partial z^2}\bigl(X(r),Z(r)\bigr) \cdot  \frac{\sigma^2}{2} \left| \nabla_{x} \Psi\bigl(r,X(r)\bigr) \right|^{2}  dr \right] \\
	&+ \int_{0}^{t} \Mean_{s}\left[ \mathbf{1}_{(0,\infty)}(Z(r))\cdot \sigma^{2} \left\langle \nabla_{x} \frac{\partial \phi}{\partial z}\bigl(X(r),Z(r)\bigr), \nabla_{x} \Psi\bigl(r,X(r)\bigr) \right\rangle  dr \right].
\end{align*}

Let $\check{\rho}$ denote the flow of conditional distributions of $(X,Z)$ given $S$; that is,
\[
\check{\rho}_{s}(t) \doteq \Prb_{s} \circ (X(t),Z(t))^{-1},\quad t\in [0,T], \quad s\in \bS.
\]
Then, in view of \eqref{EqMarkovTestfnct}, for all test functions $\phi \in \mathbf{C}^{2}_{c}(\bR^{d}\times \bR)$, all $t\in [0,T)$,
\begin{align}\label{EqKolmogorovFs}
	\int_{\bR^{d+1}} &\phi(x,z) \check{\rho}_{s}(t;dx,dz) = \int_{\bR^d} \phi\bigl(x,\Psi(0,x)\vee 0 \bigr) \nu(dx) \\
	&+ \int_{0}^{t} \int_{\bR^{d+1}} \left( \left\langle b_{s}(r,x,z), \nabla_{x} \phi(x,z) \right\rangle + \frac{\sigma^2}{2} \Delta_{x} \phi(x,z) \right) \check{\rho}_{s}(r;dx,dz)\, dr \notag\\
	&+ \int_{0}^{t} \int_{\bR^{d+1}} \mathbf{1}_{(0,\infty)}(z) \cdot \frac{\partial \phi}{\partial z}(x,z)\cdot  \left\langle b_{s}(r,x,z), \nabla_{x} \Psi(r,x) \right\rangle \check{\rho}_{s}(r;dx,dz)\, dr \notag\\
	&+ \int_{0}^{t} \int_{\bR^{d+1}} \mathbf{1}_{(0,\infty)}(z) \cdot \frac{\partial \phi}{\partial z}(x,z)\cdot \left(\frac{\partial}{\partial r}\Psi(r,x) + \frac{\sigma^2}{2} \Delta_{x} \Psi(r,x) \right) \check{\rho}_{s}(r;dx,dz)\, dr \notag\\
	&+ \int_{0}^{t} \int_{\bR^{d+1}} \mathbf{1}_{(0,\infty)}(z) \cdot \frac{\partial^{2} \phi}{\partial z^2}(x,z) \cdot  \frac{\sigma^2}{2} \left| \nabla_{x} \Psi(r,x) \right|^{2} \check{\rho}_{s}(r;dx,dz)\, dr \notag\\
	&+ \int_{0}^{t} \int_{\bR^{d+1}} \mathbf{1}_{(0,\infty)}(z) \cdot \sigma^{2} \left\langle \nabla_{x} \frac{\partial \phi}{\partial z}(x,z), \nabla_{x} \Psi(r,x) \right\rangle \check{\rho}_{s}(r;dx,dz)\, dr. \notag
\end{align}
At terminal time, $Z(T)=0$ by construction, whereas $X$ is continuous. The terminal distribution is therefore determined by
\begin{equation}\label{EqKolmogorovTerminal}
	\int_{\bR^{d+1}}\phi(x,z)\check{\rho}_s(T;dx,dz)
	= \lim_{t\uparrow T} \int_{\bR^{d+1}}\phi(x,0)\check{\rho}_s(t;dx,dz),
	\qquad \phi\in\mathbf{C}^{2}_{c}(\bR^d\times\bR).
\end{equation}
This relation accounts for the possible jump of $Z$ at time $T$ on the event $\{\tau=\infty\}$.

Now, suppose that the solutions $V_{s,\rho}$ and $\hat{V}_{\rho}$ of the Hamilton-Jacobi-Bellman equations \eqref{EqHJBVs} and \eqref{EqHJBVhat} are sufficiently regular. 
The feedback function $b$ will then be given by a measurable selection of minimizers:
\begin{equation} \label{EqOptFeedback}
	b_{s}(t,x,z) \in
	\begin{cases}
		\argmin_{a \in \Gamma}\left\{ \tilde{f}_{\rho}(t, x, a) + \left\langle a, \nabla_{x}\hat{V}_{\rho}(t,x) \right\rangle \right\} &\text{if } z > 0, \\
		\argmin_{a \in \Gamma}\left\{ f\bigl(t, x, s, a, \rho_{s}(t)\bigr) + \left\langle a, \nabla_{x}V_{s,\rho}(t,x) \right\rangle \right\} &\text{otherwise},
	\end{cases}
\end{equation}
while the function $\Psi$, which defines the continuation region, is given by
\begin{equation} \label{EqOptPsi}
	\Psi(t,x) = G_{\rho}(t,x) - \hat{V}_{\rho}(t,x).
\end{equation}
The mean field game system is therefore determined by the Hamilton-Jacobi-Bellman equations \eqref{EqHJBVs} and \eqref{EqHJBVhat} for the value functions $V_{s,\rho}$, $s\in \bS$, and $\hat{V}_{\rho}$, respectively, and the Kolmogorov forward equation \eqref{EqKolmogorovFs} together with \eqref{EqKolmogorovTerminal} for the flow of conditional distributions $\check{\rho}$ with drift coefficients $b_s$, $s\in \bS$, given by \eqref{EqOptFeedback} and function $\Psi$ given by \eqref{EqOptPsi}. In equilibrium, the conditional flow of measures $\rho$ relates to $\check{\rho}$ according to
\begin{align} \label{EqRhoCheckRho}
	& \rho_{s}\bigl(t; B\times \{0\}\bigr) = \check{\rho}_{s}\bigl(t;B\times (0,\infty) \bigr), & && \\
	\nonumber & \rho_{s}\bigl(t; B\times \{1\}\bigr) = \check{\rho}_{s}\bigl(t;B \times (-\infty,0] \bigr), & & B \in \cB(\bR^{d}),\; t\in [0,T],\; s\in \bS.&
\end{align}
Thus, we can see the value functions $V_{s,\rho}$, $s\in \bS$, $\hat{V}_{\rho}$ as parameterized by the flow of conditional distributions $\check{\rho}$ instead of the conditional flow of measures $\rho$.

Notice that the proportion of players who at time $t$ do \emph{not} know the value of $S$ will in equilibrium be given by
\[
\rho_{s}\bigl(t; \bR^{d} \times \{0\}\bigr) = \check{\rho}_{s}\bigl(t;\bR^{d}\times (0,\infty) \bigr) = \Prb\left( t < \tau_{\ast} \right),
\]
where $\tau_{\ast}$ is the optimal buying time associated with a mean field game solution $\rho$.

We do not use the system of equations \eqref{EqHJBVs}, \eqref{EqHJBVhat}, \eqref{EqKolmogorovFs} together with relations \eqref{EqOptPsi} and \eqref{EqRhoCheckRho} here to state a verification theorem. In Section~\ref{SectExample} below, we will however give a simple explicit example where candidate solutions can be verified ``by hand'' based on the above mean field game system of equations.


\section{A mean field game with linear-quadratic control structure} \label{SectExample}

We now illustrate the mean field game system developed in the previous sections by means of an explicitly solvable example.
The model has a linear-quadratic control structure, while retaining a non-trivial interaction through the terminal distribution.
In particular, the example will allow us to identify regimes in which a mean field game solution is generated either by remaining uninformed or by acquiring information immediately.

Let $d=1$, $\Gamma=\bR$, and let the hidden state take values in $\bS=\{-1,1\}$, with $q_{\bS}\in[0,1]$ and $\mu=\operatorname{Rademacher}(q_{\bS})$.
The initial position has distribution $\nu=\mathcal{N}(0,\sigma_\nu^2)$, with $\sigma_\nu^2\geq0$.
We take the running cost to be quadratic in the control,
\[
f(t,x,s,a,m)=\frac12 a^2,
\]
and assume that information can be acquired at the constant cost
\[
h(t,x)=c_2\,\mathbf 1_{[0,T)}(t),
\]
with $c_2>0$. The terminal cost is given by
\[
g(x,s,m) = c_0(x-s)^2 + x\int \mathfrak g\,dm,
\]
where $c_0>0$ and $\mathfrak g:\bR\to\bR$ is bounded and continuous.
Additional symmetry assumptions on $\mathfrak g$, and eventually a specific choice of this function, will be imposed below.
As before, we identify $\mathfrak g$ with its natural extension to $\bR\times\{0,1\}$.

The choice of the reference state $s_0$ is immaterial in this example; we fix $s_0=-1$.
Let $\rho$ be a conditional flow of measures and, for $s\in\bS$, define
\begin{equation}\label{def:csrho}
	c_{s,\rho} \doteq \int_{\bR\times\{0,1\}} \mathfrak g(y)\,\rho_s(T;dy).
\end{equation}
Thus, for a fixed flow $\rho$, the mean field interaction enters the terminal cost only through the two scalar quantities $c_{s,\rho}$.

In the notation of Section~\ref{SectMixedCP}, the running cost before information acquisition and the averaged full-information value are therefore
\begin{align*}
	& \tilde f_\rho(t,x,a)=\frac12a^2,&
	&\tilde V_\rho(t,x) = q_{\bS}V_{1,\rho}(t,x) + (1-q_{\bS})V_{-1,\rho}(t,x),&
\end{align*}
and
\[
G_\rho(t,x)=h(t,x)+\tilde V_\rho(t,x).
\]
For $s\in\{-1,1\}$, the full-information value function $V_{s,\rho}$ is characterized by the HJB equation~\eqref{EqHJBVs}, with terminal condition
\[
V_{s,\rho}(T,x) = c_0(x-s)^2+c_{s,\rho}x.
\]
Since the running cost is quadratic and the control enters the state equation linearly, the corresponding reduced Hamiltonian is
\[
H_{s,\rho}(t,x,p) = \inf_{a\in\bR} \left\{ \frac12a^2+ap \right\} = -\frac12p^2.
\]
The unique minimizer is \(a^\ast=-p\); hence, whenever \(V_{s,\rho}\) is differentiable, the associated optimal feedback control is
\[
u_{s,\rho}(t,x)=-\partial_xV_{s,\rho}(t,x).
\]

Due to the linear-quadratic structure, the value functions $V_{s,\rho}$, $s\in \{-1,1\}$, are given explicitly by
\begin{equation} \label{EqExampleVsfull} 
	V_{s,\rho}(t,x) = \varphi^{(2)}(t)\cdot x^{2} + \varphi^{(1)}_{s,\rho}(t)\cdot x + \varphi^{(0)}_{s,\rho}(t),
\end{equation}
where
\begin{align*}
	\varphi^{(2)}(t) &\doteq \frac{c_{0}}{1+2c_{0}(T-t)},
	\qquad \varphi^{(1)}_{s,\rho}(t) \doteq \frac{c_{s,\rho} - 2c_{0}\cdot s}{1+2c_{0}(T-t)}, \\
	\varphi^{(0)}_{s,\rho}(t) &\doteq  \frac{\sigma^2}{2} \log\bigl( 1+2c_{0}(T-t)\bigr) + \frac{c_{0}\cdot s^{2} + (T-t) (2c_{0}\cdot s \cdot c_{s,\rho} - c_{s,\rho}^{2}/2)}{1+2c_{0}(T-t)}.
\end{align*}
The fact that the function given by \eqref{EqExampleVsfull}, for $s\in \{-1,1\}$, solves the \mbox{HJB} equation \eqref{EqHJBVs} with $d=1$, terminal costs $g(.,s,\rho_{s}(T))$ and reduced Hamiltonian $H_{s,\rho}$ as above can be checked directly by computing partial derivatives. Notice that it is a classical solution. A standard verification argument based on It{\^o}'s formula then shows that equation \eqref{EqExampleVsfull} gives, indeed, the full information value function $V_{s,\rho}$.

It follows that
\begin{equation} \label{EqExampleVtilde}
	\tilde{V}_{\rho}(t,x) = \varphi^{(2)}(t)\cdot x^{2} + \tilde{\varphi}^{(1)}_{\rho}(t)\cdot x + \tilde{\varphi}^{(0)}_{\rho}(t),
\end{equation}
where $\varphi^{(2)}$ is as defined above and $\tilde{\varphi}^{(i)}_{\rho} \doteq q_{\bS}\varphi^{(i)}_{1,\rho} + (1-q_{\bS})\varphi^{(i)}_{-1,\rho}$, $i\in \{0,1\}$.

We now restrict attention to a symmetric setting. Besides simplifying the value function, symmetry reduces the mean field consistency condition to a scalar fixed-point problem.

\paragraph{Symmetric case.} Choose $q_{\bS} = 1/2$, so that $\mu = \text{Rademacher}(1/2)$. Take $\mathfrak{g}$ anti-symmetric in the sense that
\begin{equation} \label{EqExampleAntiSymmetry}
	-\mathfrak{g}(x) = \mathfrak{g}(-x) \quad\text{for all }x\in \bR,
\end{equation}
and consider only flows of measures $\rho$ which are symmetric in the sense that
\begin{equation} \label{EqExampleSymmetry}
	\rho_{1}(T;B\times \{0,1\}) = \rho_{-1}(T;(-B)\times \{0,1\}) \text{ for all }B\in \cB(\bR).
\end{equation}
Then, by \eqref{def:csrho},
\[
-c_{-1,\rho} = 
\int_{\bR\times \{0,1\}} \mathfrak{g}(-y)\, \rho_{-1}(T;dy) = \int_{\bR\times \{0,1\}} \mathfrak{g}(y)\, \rho_{1}(T;dy) = c_{1,\rho}.
\]
By symmetry and \eqref{EqExampleVtilde}, we then have
\begin{equation} \label{EqExampleVtildesym}
	\tilde{V}_{\rho}(t,x) = \varphi^{(2)}(t)\cdot x^{2} + \tilde{\varphi}^{(0)}_{\rho}(t),
\end{equation}
where $\tilde{\varphi}^{(1)}$ has vanished and $\tilde{\varphi}^{(0)}$ becomes
\[
\tilde{\varphi}^{(0)}_{\rho}(t) = \frac{\sigma^2}{2} \log\bigl( 1+2c_{0}(T-t)\bigr) + \frac{c_{0} + (T-t) (2c_{0}c_{1,\rho} - c_{1,\rho}^{2}/2)}{1+2c_{0}(T-t)}.
\]
Recall that the information costs are constant before terminal time. Hence,
\[
G_{\rho}(t,x) = \tilde{V}_{\rho}(t,x) + h(t,x) = \begin{cases}
	\tilde{V}_{\rho}(T,x) &\text{if } t = T,\\
	\tilde{V}_{\rho}(t,x) + c_{2} &\text{if } t\in  [0,T).
\end{cases}
\]
Let $\tilde{v}$ be the unique classical solution (of sub-exponential growth) to the \mbox{HJB} equation
\begin{equation*}
	\begin{cases}
		\tilde{v}(T, x) = G_{\rho}(T,x), & x \in \bR, \\
		\frac{\partial}{\partial t} \tilde{v}(t, x) -\frac{1}{2} \left(\nabla_{x} \tilde{v}(t,x) \right)^{2} + \frac{\sigma^{2}}{2} \Delta_{x} \tilde{v}(t, x) = 0,  &(t, x) \in[0, T) \times \bR.
	\end{cases}
\end{equation*}
Since
\[
G_{\rho}(T,.) = \tilde{V}_{\rho}(T,.) = c_{0}\left(x^{2} + 1\right),
\]
using a linear-quadratic ansatz (or, alternatively, the Hopf--Cole transformation), we find 
\[
\tilde{v}(t,x) = \varphi^{(2)}(t)\cdot x^{2} + c_{0} + \frac{\sigma^2}{2} \log\bigl( 1+2c_{0}(T-t)\bigr).
\]
Notice that $\tilde{v}$ does not depend on the conditional flow of measures $\rho$, as the terminal condition does not depend on $\rho$. It follows that $G_{\rho}(T,.) - \tilde{v}(T,.) = 0$ and, for $t < T$,
\[
G_{\rho}(t,x) - \tilde{v}(t,x) = c_{2} - \frac{2(T-t) (c_{0} - c_{1,\rho}/2)^{2}}{1+2c_{0}(T-t)}.
\]
Observe that the right-hand side above does not depend on $x\in \bR$. Denote it by 
\[
\psi_{\rho}(t)\doteq c_{2} - \frac{2(T-t) (c_{0} - c_{1,\rho}/2)^{2}}{1+2c_{0}(T-t)},\quad t\in [0,T).
\]
The function $\psi_{\rho}$ is continuous and non-decreasing, with $\lim_{t\to T-}\psi_{\rho}(t) = c_{2} > 0$. 
Consequently, $\psi_{\rho}(t)>0$ for all $t \in [0,T)$ if and only if $\psi_{\rho}(0) > 0$, which is equivalent to 
\begin{equation} \label{EqExampleNotBuy}
	2c_{0} - 2\sqrt{\frac{c_{2}(1+2c_{0}T)}{2T}} < c_{1,\rho} < 2c_{0} + 2\sqrt{\frac{c_{2}(1+2c_{0}T)}{2T}}.
\end{equation}

The monotonicity of $\psi_{\rho}$ implies that delaying the purchase of information cannot improve on remaining uninformed until terminal time or buying immediately.
Consequently, $\hat{V}_{\rho}(t,x) = \min\{G_{\rho}(t,x),\tilde{v}(t,x)\}$, so the continuation region $D = D_{\rho}$ is determined by the sign of $\psi_{\rho}$ and is given by
\[
D = \left\{\begin{array}{lll}
	[0,T) \times \bR &\text{if \eqref{EqExampleNotBuy} holds}, & \textit{remain uninformed}\\
	(t_{\ast},T) \times \bR &\text{otherwise}, & \textit{buy immediately if starting at }t\leq t_{\ast}
\end{array}\right.
\]
where $t_{\ast} = t_{\ast}(\rho)$ is the unique zero of $\psi_{\rho}$ in $[0,T)$ when \eqref{EqExampleNotBuy} does not hold.
In the first case, the representative player uses the feedback control associated with $\tilde{v}$, 
\begin{equation} \label{EqExampleFeedbackNoInfo}
	\hat{u}(t,x)\doteq -\frac{\partial}{\partial x} \tilde{v}(t,x) = -2\phi^{(2)}(t)\cdot x,\quad (t,x)\in [0,T]\times \bR,
\end{equation} 
whereas in the second case, she uses the full-information feedback control,
\begin{equation} \label{EqExampleFeedbackInfo}
	u_{s,\bar{c}}(t,x)\doteq -2\phi^{(2)}(t)\cdot x - \frac{s \cdot (\bar{c} - 2c_{0})}{1+2c_{0}(T-t)},
	\quad (t,x)\in [0,T]\times \bR, \quad s\in\{-1,1\}, 
\end{equation} 
where $\bar{c} = c_{1,\rho}$ and, by symmetry, $c_{s,\rho}=s\bar c$.

In the notation of Section~\ref{SectMFGSystem}, let $\hat{X}$ be the unique solution to \eqref{EqMarkovSDEhat} with $\hat{u}$ given by~\eqref{EqExampleFeedbackNoInfo}. For $s\in \{-1,1\}$, $\bar{c}\in \bR$, let $\bar{X}_{s,\bar{c}}$ be the solution to \eqref{EqMarkovSDEknown} with $u_{s} = u_{s,\bar{c}}$ according to \eqref{EqExampleFeedbackInfo} and stopping time $\bar{\tau} = 0$ whenever \eqref{EqExampleNotBuy} does not hold. The initial position $\xi$ in \eqref{EqMarkovSDEhat}, \eqref{EqMarkovSDEknown} is independent of the driving Wiener process and has distribution $\nu$, which here is assumed to be Gaussian with mean zero and variance $\sigma_{\nu}^{2}$. As the feedback strategies here are affine-linear and the Wiener noise additive, equations \eqref{EqMarkovSDEhat} and \eqref{EqMarkovSDEknown} can be solved explicitly; see, for instance, Section~5.6.C in \citet[pp.\ 360--361]{karatzasshreve91}. In particular, $\hat{X}$, $\bar{X}_{s,\bar{c}}$, $s\in \{-1,1\}$, are Gaussian processes, with 
\begin{subequations} \label{EqExampleMeanVar}
	\begin{align}
		&\Mean\left[ \hat{X}(T) \right] = 0, &
		& \Mean_{s}\left[ \bar{X}_{S,\bar{c}}(T) \right] = \Mean\left[ \bar{X}_{s,\bar{c}}(T) \right] = \frac{s(2c_{0}-\bar{c})T}{1+2c_{0}T}, &
	\end{align}
	and, under $\Prb$ as well as $\Prb_s$, $s\in\{-1,1\}$,
	\begin{equation}
		\Var\left[\hat{X}(T)\right]
		= \Var\left[\bar{X}_{s,\bar{c}}(T)\right]
		= \frac{\sigma_{\nu}^{2}}{(1+2c_{0}T)^2}
		+ \frac{\sigma^{2}T}{1+2c_{0}T}.
	\end{equation}
\end{subequations}

We now determine sufficient conditions under which each of the two buying decisions gives rise to a symmetric solution of the mean field game.

\paragraph{Case 1: No information acquisition.}
Suppose that no agent buys information before terminal time and that all agents use the feedback control $\hat{u}$ in \eqref{EqExampleFeedbackNoInfo}.
Define
\[
\rho_s(t) \doteq \Prb_s\circ\left(\hat{X}(t),\mathbf{1}_{\{T\}}(t)\right)^{-1}, \quad t\in[0,T], \quad s\in\{-1,1\}.
\]
Since $\hat{X}$ is independent of $S$, the conditional flows $\rho_1$ and $\rho_{-1}$ coincide.
Moreover, $\hat{X}(T)$ has a symmetric Gaussian distribution by \eqref{EqExampleMeanVar}, so the antisymmetry of $\mathfrak{g}$ implies $c_{1,\rho}=c_{-1,\rho}=0$.
Consequently, the consistency condition is satisfied, and the strategy is optimal when \eqref{EqExampleNotBuy} holds with $c_{1,\rho}=0$, that is, $c_0 < K_1 \sqrt{c_2}$ where we set
\[
K_1 \doteq \sqrt{\frac{1+2c_0T}{2T}}.
\]
This is equivalent to
\begin{align} \label{EqExampleNoBuyCondition}
	&\text{either} && c_0\leq c_2 &&\text{or} && c_0>c_2 \text{ and } T<\frac{c_2}{2c_0(c_0-c_2)}.
\end{align}
Thus, under \eqref{EqExampleNoBuyCondition}, $\rho$ is a strong solution of the mean field game in which no agent buys information before terminal time.
This occurs when information is sufficiently expensive or, otherwise, when the time horizon is sufficiently short.

\paragraph{Case 2: Immediate information acquisition.}
Suppose now that all agents buy information at time zero.
We seek a symmetric equilibrium in which the mean-field interaction coefficient conditional on $S=1$ is $\bar{c}$, and hence the coefficient conditional on $S=-1$ is $-\bar{c}$.
The parameter $\bar{c}$ is initially arbitrary and will be determined by the consistency condition.
For immediate information acquisition to be optimal, \eqref{EqExampleNotBuy} must fail with $\bar{c}$ in place of $c_{1,\rho}$, that is,
\begin{align} \label{EqExampleBuyCondition1}
	&\bar{c} \leq 2c_0-2K_1\sqrt{c_2} &&\text{or} &&\bar{c} \geq 2c_0+2K_1\sqrt{c_2}.
\end{align}
For such a $\bar{c}$, define
\[
\rho_1(t) \doteq \Prb_1\circ\left(\bar{X}_{1,\bar{c}}(t),1\right)^{-1}, \qquad \rho_{-1}(t) \doteq \Prb_{-1}\circ\left(\bar{X}_{-1,\bar{c}}(t),1\right)^{-1}, \quad t\in[0,T].
\]
These flows are symmetric, and consistency reduces to
\begin{equation} \label{EqExampleBuyCondition2}
	\Mean\left[\mathfrak{g}\bigl(\bar{X}_{1,\bar{c}}(T)\bigr)\right] = \bar{c}.
\end{equation}
Thus, finding an equilibrium with immediate information acquisition amounts to finding a value of $\bar{c}$ satisfying both \eqref{EqExampleBuyCondition1} and \eqref{EqExampleBuyCondition2}.
To obtain an explicit sufficient condition for existence, choose
\[
\mathfrak{g}(x) \doteq x\cdot e^{-x^2}, \quad x\in\bR.
\]
This function is bounded, continuous, and antisymmetric.
Since $\bar{X}_{1,\bar{c}}(T)$ is Gaussian with mean and variance given by \eqref{EqExampleMeanVar}, a direct computation yields
\[
\Mean\left[\mathfrak{g}\bigl(\bar{X}_{1,\bar{c}}(T)\bigr)\right] = K_2(2c_0-\bar{c})\exp\left(-K_3(2c_0-\bar{c})^2\right),
\]
where
\begin{align*}
	K_2 &\doteq \frac{T(1+2c_0T)^2}{\left((1+2c_0T)^2+2\sigma^2T(1+2c_0T)+2\sigma_\nu^2\right)^{3/2}},\\
	K_3 &\doteq \frac{T^2}{(1+2c_0T)^2+2\sigma^2T(1+2c_0T)+2\sigma_\nu^2}.
\end{align*}
Consequently, the consistency condition~\eqref{EqExampleBuyCondition2} becomes the fixed-point equation
\[
\bar{c} = K_2(2c_0-\bar{c})\exp\left(-K_3(2c_0-\bar{c})^2\right).
\]
Let $\phi(\bar{c})$ denote the right-hand side of this equation and set
\[
\bar{c}_{\pm} \doteq 2c_0\pm 2K_1\sqrt{c_2}.
\]
Condition \eqref{EqExampleBuyCondition1} requires $\bar{c}\in\bR\setminus(\bar{c}_-,\bar{c}_+)$.
Since $\phi(\bar{c})<0$ for $\bar{c}\geq\bar{c}_+$, no fixed point can lie in $[\bar{c}_+,\infty)$.
Moreover, $\phi$ is continuous and tends to zero as $\bar{c}\to-\infty$.
Therefore, a sufficient condition for the existence of a fixed point in $(-\infty,\bar{c}_-]$ is
$\phi(\bar{c}_-) \leq \bar{c}_-$.
Substituting the expressions for $\phi$ and $\bar{c}_-$, this becomes
\begin{equation} \label{EqExampleBuySufficient}
	K_1K_2\sqrt{c_2}\exp\left(-4K_1^2K_3c_2\right)
	\leq c_0-K_1\sqrt{c_2}.
\end{equation}
Notice that \eqref{EqExampleBuySufficient} holds if $c_2$ is sufficiently small compared to $c_0$ and $T$.
Consequently, under \eqref{EqExampleBuySufficient}, there exists $\bar{c}\in(-\infty,\bar{c}_-]$ that satisfies both \eqref{EqExampleBuyCondition1} and \eqref{EqExampleBuyCondition2}.
The corresponding conditional flow of measures $\rho$ is therefore a strong solution of the mean field game in which all agents buy information at time zero.


\section{An explicit acquisition boundary with mean field interaction}\label{SectExampleMeanFieldBoundary}

We give an example with a nontrivial dependence on the population distribution for which the optimal acquisition region has an explicit boundary.
The individual optimization problem is first parametrized by the magnitude of the informed drift, while the equilibrium value of this parameter is subsequently determined by a scalar consistency equation.

\paragraph{Model and informed values.}
Let $d=1$, $\Gamma=\bR$, and assume $\sigma>0$.
Let the hidden state take values in $\bS=\{-1,1\}$, with $\mu=\operatorname{Rademacher}(1/2)$.
The initial position has distribution
\[
\nu=\frac12\left(\delta_{-x_0}+\delta_{x_0}\right),
\qquad x_0>0.
\]
To specify the mean field interaction, fix $\lambda$ with $0<\lambda<\min\{1,1/T\}$, and define
\[
\vartheta(y)\doteq\frac{y^2}{1+y^2},\qquad y\in\bR.
\]
As before, we identify $\vartheta$ with its natural extension to $\bR\times\{0,1\}$.
The running and terminal costs are
\begin{align}\label{EqMFBoundaryCosts}
	&f(t,x,s,a,m)=\frac{1}{2}a^2,&
	&g(x,s,m)=-s\left(1-\lambda\int\vartheta\,dm\right)x.&
\end{align}
Since $0\leq\vartheta\leq1$ and $\lambda<1$, the mean field interaction affects the strength but not the direction of the informed agent's incentive to move. Indeed, the representative agent will move to the right when $s=1$ and to the left when $s=-1$.
The boundedness and Lipschitz continuity of $\vartheta$ imply that these costs satisfy \hypref{ARunTermCosts}, \hypref{ACoercivity}, and \hypref{ALocLipMeasure}, with $\mathrm{d}_{\cP}$ chosen as the bounded Lipschitz metric.

Fix $q\in[1-\lambda,1]$, and consider a conditional flow of measures $\rho=(\rho_s)_{s\in\bS}$ such that
\begin{equation}\label{EqMFBoundaryMoment}
	\int\vartheta\,d\rho_1(T)
	=\int\vartheta\,d\rho_{-1}(T)
	=\frac{1-q}{\lambda}.
\end{equation}
The terminal cost for an informed agent with $S=s$ is then $g\bigl(X(T),s,\rho_s(T)\bigr)=-sqX(T)$.
\break 
For $(t_0,x)\in[0,T]\times\bR$ and $\alpha\in\cA_{t_0}$, writing $X\doteq X^{t_0,x,\alpha}$ and completing the square gives
\[
\Mean\left[\int_{t_0}^T\frac{\alpha(t)^2}{2}dt-sqX(T)\right] 
=-sqx-\frac{q^2}{2}(T-t_0) 
+\frac{1}{2}\Mean\left[\int_{t_0}^T\bigl(\alpha(t)-sq\bigr)^2dt\right].
\]
Thus the informed value and an optimal feedback are
\begin{equation}\label{EqMFBoundaryInformedValue}
	V_{s,\rho}(t,x)=-sqx-\frac{q^2}{2}(T-t),\qquad
	u_s(t,x)=sq.
\end{equation}
Hence, after information acquisition, the sign of $S$ determines the direction of the optimal motion.
Before acquisition, the two directional incentives cancel by symmetry, since
\[
\frac{1}{2}\bigl(V_{1,\rho}(t,x)+V_{-1,\rho}(t,x)\bigr)
=-\frac{q^2}{2}(T-t).
\]
The motion before acquisition is therefore driven by the option to acquire information rather than by a directional terminal incentive.
Accordingly, for an information price $h$ to be specified below,
\[
G_\rho(t,x)=h(t,x)+\tilde V_\rho(t,x)
=h(t,x)-\frac{q^2}{2}(T-t).
\]
Since $\tilde f_\rho(t,x,a)=a^2/2$ and $H_\rho(t,x,p)=-p^2/2$, the variational inequality \eqref{EqHJBVhat} associated with \eqref{ExVhat} becomes
\[
\min\left\{G_\rho-v,\,
\partial_tv+\frac{\sigma^2}{2}\partial_{xx}v-\frac{1}{2}(\partial_xv)^2\right\}=0,
\qquad v(T,x)=0.
\]
The Hopf--Cole transformation $v=-\sigma^2\log U$ gives
\begin{equation}\label{EqMFBoundaryTransformedVI}
	\max\left\{\widehat G_\rho-U,\,
	\partial_tU+\frac{\sigma^2}{2}\partial_{xx}U\right\}=0,
	\qquad U(T,x)=1,    
\end{equation}
where
\[
\widehat G_\rho(t,x)\doteq
\exp\left(-\frac{G_\rho(t,x)}{\sigma^2}\right)
\]
is the transformed obstacle.

\paragraph{Information price and acquisition boundary.}
Fix $c,\delta>0$, and assume that\break $x_0>c\sigma\sqrt T$.
For $t<T$, introduce the level-set variable $z(t,x)\doteq |x|/(\sigma\sqrt{T-t})$, whose level sets are the curves $x(t,z)=\pm z\, \sigma\sqrt{T-t}$.
Motivated by the construction of~\cite{dauriaferriero20}, in the spirit of the explicit formulas derived in~\cite{shepp69}, consider the profile
\[
F_0(z)\doteq(1+z^2)\Phi(-z)-z\Phi'(z),\qquad z\geq0,
\]
where $\Phi$ denotes the standard normal distribution function.
It satisfies
\[
F_0''(z)+zF_0'(z)-2F_0(z)=0,\qquad z>0.
\]
Accordingly, a function of the form $1+\varepsilon(T-t)F_0(z(t,x))$ satisfies the heat equation wherever it is twice differentiable.

We then single out the level $c$ by retaining $F_0$ for $z>c$ while replacing it on $[0,c]$ by the quadratic extension that matches its value and first derivative at $c$ and has zero derivative at the origin.
Thus, define
\begin{equation}\label{EqMFBoundaryF}
	F_c(z)\doteq
	\begin{cases}
		(1+z^2)\Phi(-c)-z^2\Phi'(c)/c,&0\leq z\leq c,\\
		F_0(z),&z>c.
	\end{cases}
\end{equation}
Choose $\varepsilon>0$ such that
\begin{equation}\label{EqMFBoundaryEpsilon}
	\varepsilon\leq\min\left\{\frac{1}{2T\delta},\,
	\frac{(1-\lambda)^2}{4\sigma^2\Phi(-c)}\right\},
\end{equation}
and define
\begin{equation}\label{EqMFBoundaryUH}
	U(t,x)\doteq1+\varepsilon(T-t)F_c(z(t,x)),\qquad t<T,
	\qquad U(T,x)\doteq1.
\end{equation}
The modification of $F_0$ on $[0,c]$ yields a strict differential inequality on the side $z<c$, while the heat equation continues to hold for $z>c$.

The contact with the obstacle is determined separately.
Define
\[
D_c(z)\doteq\delta\bigl(1-e^{-(z-c)_+^2}\bigr),\qquad z\geq0,
\]
where $(\cdot)_+ \doteq \max\{0,\cdot\}$.
Thus $D_c=0$ on $[0,c]$, $D_c>0$ on $(c,\infty)$, and $0\leq D_c\leq\delta$.
For the fixed parameter $q$, set
\begin{equation}\label{EqMFBoundaryPrice}
	h(t,x)\doteq\frac{q^2}{2}(T-t)
	-\sigma^2\log\bigl(U(t,x)-\varepsilon(T-t)D_c(z(t,x))\bigr),\qquad t<T,
\end{equation}
and $h(T,x)\doteq0$.
Equivalently,
\begin{equation}\label{EqTransformedObstacle}
	\widehat G_\rho(t,x) = U(t,x)-\varepsilon(T-t)D_c(z(t,x)),\qquad t<T,
\end{equation}
so that $\widehat G_\rho$ coincides with $U$ for $z\leq c$ and lies strictly below it for $z>c$.

\begin{figure}[htbp]
	\centering
	
	\begingroup
	\def\W{5.0}
	\def\H{1.75}
	\def\cval{1.0}
	
	\newcommand{\MFaxes}{%
		\draw[->] (0,-\H) -- (0,\H) node[above left] {$x$};
		\draw[->] (0,0) -- (\W+0.4,0) node[right] {$t$};
		\node[below left] at (0,0) {$0$};
		\node[below] at (\W+0.1,0) {$T$};
	}
	
	\newcommand{\MFcurve}[2][]{%
		\draw[#1,domain=0:1,samples=80,smooth,variable=\s]
		plot ({\W*\s},{#2*sqrt(1-\s)});
		\draw[#1,domain=0:1,samples=80,smooth,variable=\s]
		plot ({\W*\s},{-#2*sqrt(1-\s)});
	}
	
	\newcommand{\MFinside}[1]{%
		\fill[#1]
		(0,\cval)
		plot[domain=0:1,samples=80,smooth,variable=\s]
		({\W*\s},{\cval*sqrt(1-\s)})
		--
		plot[domain=1:0,samples=80,smooth,variable=\s]
		({\W*\s},{-\cval*sqrt(1-\s)})
		-- cycle;
	}
	
	\newcommand{\MFoutsidepattern}{%
		\fill[pattern=north east lines,pattern color=gray!45]
		(0,\cval)
		-- (0,\H)
		-- (\W,\H)
		-- (\W,0)
		-- plot[domain=1:0,samples=80,smooth,variable=\s]
		({\W*\s},{\cval*sqrt(1-\s)})
		-- cycle;
		
		\fill[pattern=north east lines,pattern color=gray!45]
		(0,-\cval)
		-- (0,-\H)
		-- (\W,-\H)
		-- (\W,0)
		-- plot[domain=1:0,samples=80,smooth,variable=\s]
		({\W*\s},{-\cval*sqrt(1-\s)})
		-- cycle;
	}

	\begin{subfigure}[t]{0.48\textwidth}
		\centering
		\resizebox{\linewidth}{!}{%
			\begin{tikzpicture}[font=\small,>=Latex]
				\MFaxes
				
				\MFcurve[gray!40,densely dotted]{0.55}
				\MFcurve[gray!60,dashed]{1.00}
				\MFcurve[gray!40,densely dotted]{1.45}
				\MFcurve[gray!30,densely dotted]{1.75}
				
				\node[anchor=west,font=\footnotesize] at (0.25,1.53) {$z=1.75$};
				\node[anchor=west,font=\footnotesize] at (0.75,1.05) {$z=1.45$};
				\node[anchor=west,font=\footnotesize] at (1.35,0.61) {$z=1$};
				\node[anchor=west,font=\footnotesize] at (1.90,0.18) {$z=0.55$};
				
				\node[align=center,font=\footnotesize] at (3.25,-1.28)
				{$U_0=1+\varepsilon(T-t)F_0(z)$\\[1pt]
					$U_{0,t}+\frac{\sigma^2}{2}U_{0,xx}=0$};
		\end{tikzpicture}}
		\caption{The profile $F_0$ generates the family of level curves $z=\mathrm{const.}$; no member of the family is distinguished.}
		\label{FigMFBoundaryConstructionF0}
	\end{subfigure}
	\hfill
	\begin{subfigure}[t]{0.48\textwidth}
		\centering
		\resizebox{\linewidth}{!}{%
			\begin{tikzpicture}[font=\small,>=Latex]
				\MFaxes
				\MFinside{gray!22}
				
				\MFcurve[gray!30,densely dotted]{0.55}
				\MFcurve[black,very thick]{1.00}
				\MFcurve[gray!30,densely dotted]{1.45}
				\MFcurve[gray!25,densely dotted]{1.75}
				
				\node[anchor=west,font=\footnotesize] at (0.25,1.12) {$z=c$};
				
				\node[align=center,font=\footnotesize] at (2.15,0)
				{$z<c$\\[0.15cm]
					$U_t+\frac{\sigma^2}{2}U_{xx}<0$\\
					buy};
				
				\node[align=center,font=\footnotesize] at (3.10,-1.22)
				{$z>c:\quad F_c=F_0$\\
					$U_t+\frac{\sigma^2}{2}U_{xx}=0$};
		\end{tikzpicture}}
		\caption{The modification $F_c$ singles out the level $z=c$.
			For $z<c$ the differential inequality is strict, whereas for $z>c$ the continuation equation holds.}
		\label{FigMFBoundaryConstructionFc}
	\end{subfigure}
	
	\vspace{0.8em}
	
	\begin{subfigure}[t]{0.48\textwidth}
		\centering
		\resizebox{\linewidth}{!}{%
			\begin{tikzpicture}[font=\small,>=Latex]
				\MFaxes
				\MFinside{gray!22}
				
				\MFcurve[black,very thick]{1.00}
				\node[anchor=west,font=\footnotesize] at (0.25,1.12) {$z=c$};
				
				\node[align=center,font=\footnotesize] at (1.70,0.02)
				{$G_\rho=\hat V_\rho$\\
					buy};
				
				\node[align=center,font=\footnotesize] at (3.20,1.18)
				{$G_\rho=\hat V_\rho$\\
					buy or wait};
				
				\node[align=center,font=\footnotesize] at (3.20,-1.18)
				{$G_\rho=\hat V_\rho$\\
					buy or wait};
		\end{tikzpicture}}
		\caption{Without an obstacle gap, the buying region is unchanged, while in the white exterior immediate acquisition and continuation are both optimal.}
		\label{FigMFBoundaryConstructionNoGap}
	\end{subfigure}
	\hfill
	\begin{subfigure}[t]{0.48\textwidth}
		\centering
		\resizebox{\linewidth}{!}{%
			\begin{tikzpicture}[font=\small,>=Latex]
				
				\MFoutsidepattern
				
				\MFaxes
				\MFinside{gray!22}
				
				\MFcurve[black,very thick]{1.00}
				
				\node[anchor=west,font=\footnotesize] at (0.25,1.12) {$z=c$};
				
				\node[align=center,font=\footnotesize] at (1.70,0.02)
				{$\begin{array}{l}
						D_c=0\\
						G_\rho = \hat V_\rho
					\end{array} \text{buy}$};
				
				\node[font=\footnotesize] at (3.20,1.30)
				{$\begin{array}{l}
						D_c>0\\
						G_\rho>\hat V_\rho
					\end{array} \text{continue}$};

				\node[font=\footnotesize] at (3.20,-1.30)
				{$\begin{array}{l}
						D_c>0\\
						G_\rho>\hat V_\rho
					\end{array} \text{continue}$};
				
		\end{tikzpicture}}
		\caption{The function $D_c$ leaves the gray buying region unchanged and acts only in the hatched exterior, where continuation becomes strictly preferable.}
		\label{FigMFBoundaryConstructionDc}
	\end{subfigure}
	
	\endgroup
	
	\caption{Construction of the explicit information-acquisition boundary in the normalized geometry $T=\sigma=c=1$.
		The profile $F_0$ generates the family of level curves $z=\mathrm{const.}$.
		The modification $F_c$ selects the curve $z=c$ by making waiting strictly suboptimal for $z<c$.
		Solid gray denotes the buying region throughout panels (b)--(d).
		Without an obstacle gap, immediate acquisition and continuation are both optimal in the white exterior region $z>c$.
		The function $D_c$ acts only for $z>c$, shown by hatching in panel (d), where immediate acquisition becomes strictly suboptimal.}
	\label{FigMFBoundaryConstruction}
\end{figure}

Figure~\ref{FigMFBoundaryConstruction} illustrates the construction in the normalized case $T=\sigma=c=1$, so that the selected level is $z=c=1$ and the corresponding boundary is $|x|=\sqrt{1-t}$.
The figure shows how the profile $F_0$ generates a family of level curves, how the modification $F_c$ singles out the level $z=c$ by producing a strict differential inequality for $z<c$, and how the function $D_c$ creates a strict obstacle gap for $z>c$, thereby identifying the continuation region.
The parameter $q$ is fixed in the individual optimization problem.
Its equilibrium value is determined below.

\begin{lemma}\label{LemmaMFBoundaryPrice}
	The function $F_c$ in \eqref{EqMFBoundaryF} is positive, bounded, and continuously differentiable on $[0,\infty)$, with $F_c'(0)=0$ and $F_c'(z)<0$ for $z>0$.
	Its first derivative and its piecewise second derivative are bounded, and
	\[
	F_c(z)-\frac z2F_c'(z)-\frac{1}{2}F_c''(z)
	=\begin{cases}
		\Phi'(c)/c,&0<z<c,\\
		0,&z>c.
	\end{cases}
	\]
	Consequently, at every $(t,x)\in[0,T)\times\bR$ with $|x|\neq c\sigma\sqrt{T-t}$,
	\[
	\partial_tU(t,x)+\frac{\sigma^2}{2}\partial_{xx}U(t,x)
	=\begin{cases}
		-\varepsilon\Phi'(c)/c,&|x|<c\sigma\sqrt{T-t},\\
		0,&|x|>c\sigma\sqrt{T-t}.
	\end{cases}
	\]
	For every $q\in[1-\lambda,1]$, the price \eqref{EqMFBoundaryPrice} satisfies \hypref{AInfoCosts}.
\end{lemma}

\begin{proof}
	For $z>c$, the differential equation for $F_0$ gives the stated identity.
	The representation
	\[
	F_0(z)=\int_z^\infty(y-z)^2\Phi'(y)dy
	\]
	shows that $F_0$ is positive and strictly decreasing on $(0,\infty)$.
	The interior branch in \eqref{EqMFBoundaryF} matches $F_0$ and $F_0'$ at $c$ and satisfies $F_c'(0)=0$.
	Since $\Phi'(c)>c\Phi(-c)$, it is strictly decreasing on $(0,c]$.
	Consequently, $F_c$ is continuously differentiable and
	\[
	0<F_c(z)\leq F_c(0)=\Phi(-c),\qquad z\geq0.
	\]
	The identity on $(0,c)$ follows by direct differentiation of the quadratic branch.
	Boundedness of the first and piecewise second derivatives follows from the explicit formulas.
	Differentiating \eqref{EqMFBoundaryUH} gives the stated equation for $U$, with the formula at $x=0$ following from the quadratic form of the interior branch.
	
	Since $0\leq D_c\leq\delta$, condition \eqref{EqMFBoundaryEpsilon} gives, for $t<T$,
	\[
	\frac{1}{2}\leq1-\varepsilon T\delta
	\leq U(t,x)-\varepsilon(T-t)D_c(z(t,x))
	\leq1+\varepsilon(T-t)\Phi(-c).
	\]
	Using $\log(1+y)\leq y$ for $y>-1$, we obtain
	\[
	h(t,x)
	\geq(T-t)\left(\frac{q^2}{2}-\sigma^2\varepsilon\Phi(-c)\right)
	\geq\frac{(1-\lambda)^2}{4}(T-t)\geq0.
	\]
	The lower bound on the logarithm's argument also gives $h(t,x)\leq T/2+\sigma^2\log2$.
	Since $F_c$ and $D_c$ are bounded and continuous, $h$ is continuous on $[0,T)\times\bR$ and converges uniformly to zero as $t\uparrow T$.
	Thus \hypref{AInfoCosts} holds.
\end{proof}

\begin{prop}\label{PropMFBoundaryStopping}
	Suppose that \eqref{EqMFBoundaryMoment} holds and that the price is given by \eqref{EqMFBoundaryPrice} with the same parameter $q$.
	Then $V_\rho=\hat V_\rho=-\sigma^2\log U$, and an optimal feedback before acquisition is
	\begin{equation}\label{EqMFBoundaryFeedback}
		\hat u(t,x)\doteq\sigma^2\frac{\partial_xU(t,x)}{U(t,x)},\qquad t<T,
		\qquad\hat u(T,x)\doteq0.
	\end{equation}
	The buying region before the terminal time is exactly
	\begin{equation}\label{EqMFBoundaryRegion}
		\bigl\{(t,x)\in[0,T)\times\bR:|x|\leq c\sigma\sqrt{T-t}\bigr\}.
	\end{equation}
	In the notation of Section~\ref{SectMFGSystem}, the continuation region is
	\[
	D=\bigl\{(t,x)\in[0,T)\times\bR:|x|>c\sigma\sqrt{T-t}\bigr\},
	\]
	The first exit from $D$, capped at $T$, is optimal.
\end{prop}

\begin{proof}
	By Lemma~\ref{LemmaMFBoundaryPrice} and the construction~\eqref{EqTransformedObstacle} of $\widehat G_\rho$, $U$ satisfies  \eqref{EqMFBoundaryTransformedVI} with the obstacle active on the buying region \eqref{EqMFBoundaryRegion} 
	while the differential equation holds on the continuation region $D$. 
	
	Set $v\doteq-\sigma^2\log U$.
	Since $U\geq \widehat G_\rho$, we have $v\leq G_\rho$, with equality precisely on the buying region before $T$, and $v(T,x)=G_\rho(T,x)=0$.
	The derivatives $U_x$ and the piecewise derivative $U_{xx}$ are bounded, and $U_x$ is continuous across the two curves $|x|=c\sigma\sqrt{T-t}$.
	Consequently, $v_x$ and its piecewise derivative $v_{xx}$ are bounded, while $\hat u=-v_x$ is bounded and uniformly Lipschitz in $x$.
	At every point with $|x|\neq c\sigma\sqrt{T-t}$,
	\[
	v_t+\frac{\sigma^2}{2}v_{xx}-\frac{1}{2}v_x^2
	=-\frac{\sigma^2}{U}\left(U_t+\frac{\sigma^2}{2}U_{xx}\right)\geq0,
	\]
	with equality in $D$.
	For every $a\in\bR$, completing the square gives
	\[
	v_t+\frac{\sigma^2}{2}v_{xx}+av_x+\frac{a^2}{2}=\left(v_t+\frac{\sigma^2}{2}v_{xx}-\frac12v_x^2\right)+\frac12(a+v_x)^2\geq\frac12(a+v_x)^2\geq0.
	\]
	Equivalently,
	\[
	v_t+\frac{\sigma^2}{2}v_{xx}+av_x\geq-\frac{a^2}{2}.
	\]
	Therefore, along the state process associated with any admissible control $\alpha$,
	\[
	v_t(t,X_t)+\frac{\sigma^2}{2}v_{xx}(t,X_t)+\alpha_t v_x(t,X_t)\geq-\frac{\alpha_t^2}{2}.
	\]
	In the continuation region $D$, equality holds when $\alpha_t=\hat u(t,X_t)=-v_x(t,X_t)$.
	Fix $(t_0,x)\in[0,T)\times\bR$.
	For $\alpha\in\cA_{t_0}$ and $\tau\in\cT_{t_0,T}$, write $X\doteq X^{t_0,x,\alpha}$.
	Apply the generalized It\^o formula \citep[cf.\ Theorem~2.1 in][]{caideangelis23} to $v(t,X(t))$ up to $\tau\wedge(T-\gamma)$, where $0<\gamma<T-t_0$.
	There are no boundary local-time terms because $v_x$ is continuous across the two curves, and the occupation-time formula shows that $X$ spends zero Lebesgue time on them.
	The stochastic integral has zero expectation since $v_x$ is bounded.
	Letting $\gamma\downarrow0$ gives
	\begin{align*}
		v(t_0,x)
		&\leq\Mean\left[\int_{t_0}^\tau\frac{\alpha(t)^2}{2}dt+v\bigl(\tau,X(\tau)\bigr)\right]\\
		&\leq\Mean\left[\int_{t_0}^\tau\frac{\alpha(t)^2}{2}dt+G_\rho\bigl(\tau,X(\tau)\bigr)\right].
	\end{align*}
	This passage to the limit is justified by square integrability of $\alpha$, boundedness of $v$, and its uniform convergence to zero as $t\uparrow T$.
	For the feedback $\hat u=-v_x$ and the first exit from $D$, capped at $T$, both inequalities are equalities.
	Indeed, the differential expression vanishes before exit, and $v=G_\rho$ at the stopping time, including when it equals $T$.
	This proves $\hat V_\rho=v$ and optimality of the stated feedback and stopping rule.
	The buying region is therefore \eqref{EqMFBoundaryRegion}, and $V_\rho=\hat V_\rho$ follows from Theorem~\ref{ThValueEquiv}.
\end{proof}

In the notation of \eqref{EqOptPsi},
\[
\Psi(t,x)=G_\rho(t,x)-\hat V_\rho(t,x)
=\sigma^2\log\frac{U(t,x)}{U(t,x)-\varepsilon(T-t)D_c(z(t,x))},\qquad t<T,
\]
with $\Psi(T,x)=0$.
Since $D_c$ vanishes exactly on $[0,c]$,
\[
D=\{\Psi>0\}
=\bigl\{(t,x)\in[0,T)\times\bR:|x|>c\sigma\sqrt{T-t}\bigr\}.
\]
The value $\hat V_\rho$, the feedback $\hat u$, and the buying boundary do not depend on $q$ when the price and terminal moments are related by \eqref{EqMFBoundaryMoment} and \eqref{EqMFBoundaryPrice}.

\paragraph{Mean field consistency.}
Following the notation of Section~\ref{SectMFGSystem}, let $\hat X$ be the unique strong solution on $[0,T]$ of \eqref{EqMarkovSDEhat} with feedback \eqref{EqMFBoundaryFeedback} and initial law $\nu$.
Existence and uniqueness follow from the boundedness and uniform spatial Lipschitz continuity of $\hat u$.
Set
\begin{equation}\label{EqMFBoundaryTau}
	\tau\doteq T\wedge\inf\bigl\{t\in[0,T):|\hat X(t)|\leq c\sigma\sqrt{T-t}\bigr\},
\end{equation}
where $\inf\varnothing=\infty$.
The continuity of $\hat X$ and of the boundary implies that $\tau$ is a stopping time for the filtration generated by $(\xi,W)$.
Both $\hat X$ and $\tau$ are independent of $S$ and do not depend on $q$.
Moreover, $\tau>0$ almost surely since $x_0>c\sigma\sqrt T$.
For $q\in[1-\lambda,1]$, the state process obtained by using $\hat u$ up to $\tau$ and the informed control given by \eqref{EqMFBoundaryInformedValue} afterwards is
\[
X^q(t)\doteq\hat X(t\wedge\tau) + q S \cdot (t-\tau)_+  +\sigma\bigl(W(t)-W(t\wedge\tau)\bigr),\qquad t\in[0,T].
\]
Define
\begin{equation}\label{EqMFBoundaryM}
	\mathcal M(q)\doteq\Mean\bigl[\vartheta(X^q(T))\mid S=1\bigr],
	\qquad q\in[1-\lambda,1].
\end{equation}

\begin{prop}\label{PropMFBoundary}
	The scalar consistency equation
	\begin{equation}\label{EqMFBoundaryFixedPoint}
		q^*=1-\lambda\mathcal M(q^*)
	\end{equation}
	has a unique solution in $[1-\lambda,1]$, and this solution satisfies $1-\lambda<q^*<1$.
	Choose the price \eqref{EqMFBoundaryPrice} with $q=q^*$, and define
	\begin{equation}\label{EqMFBoundaryFlow}
		\rho_s(t)\doteq\Law\bigl((X^{q^*}(t),\mathbf{1}_{\{\tau\leq t\}})\mid S=s\bigr),
		\qquad s\in\{-1,1\},\quad t\in[0,T].
	\end{equation}
	Then $\rho$ is a strong solution of the mean field game in the sense of Definition~\ref{DefMFGstrong}.
	An optimal strategy acquires information at $\tau$, uses the feedback $\hat u$ before $\tau$, and uses the constant control $q^*S$ thereafter.
	The full-information and uninformed value functions are
	\begin{equation}\label{EqMFBoundaryValues}
		V_{s,\rho}(t,x)=-sq^*x-\frac{(q^*)^2}{2}(T-t),\qquad
		V_\rho(t,x)=\hat V_\rho(t,x)=-\sigma^2\log U(t,x),
	\end{equation}
	and the buying region is \eqref{EqMFBoundaryRegion}.
\end{prop}

\begin{proof}
	The symmetry of $\nu$ and the oddness of $\hat u$ imply that replacing $(\xi,W)$ by $(-\xi,-W)$ reflects $\hat X$ and leaves $\tau$ unchanged.
	Since $S$ is independent of $(\xi,W)$, the conditional law of $X^q(T)$ given $S=-1$ is the reflection of its conditional law given $S=1$.
	The evenness of $\vartheta$ therefore gives
	\[
	\Mean\bigl[\vartheta(X^q(T))\mid S=-1\bigr]=\Mean\bigl[\vartheta(X^q(T))\mid S=1\bigr]=\mathcal M(q).
	\]
	Moreover, independence of $S$ and $(\hat X,\tau,W)$ yields
	\[
	\mathcal M(q)=\Mean\left[\vartheta\bigl(\hat X(\tau)+q(T-\tau)+\sigma(W(T)-W(\tau))\bigr)\right].
	\]
	Since $|\vartheta'|\leq1$ and $0\leq T-\tau\leq T$, we have $|\mathcal M(q)-\mathcal M(\tilde q)|\leq T|q-\tilde q|$. Also, $0\leq\mathcal M(q)\leq1$, so $q\mapsto1-\lambda\mathcal M(q)$ maps $[1-\lambda,1]$ into itself and has Lipschitz constant at most $\lambda T<1$. The contraction mapping theorem gives a unique fixed point $q^*$. In fact, $\mathcal M(q)<1$ because $\vartheta(y)<1$ for every $y$. Conditional on $S=1$, using the dynamics of $\hat X$ up to $\tau$, the terminal state can also be written as
	\[
	X^q(T)=\xi+\int_0^\tau\hat u\bigl(t,\hat X(t)\bigr)dt+q(T-\tau)+\sigma W(T).
	\]
	The first three terms on the right-hand side are bounded, whereas $W(T)$ has unbounded support also conditional on $S=1$, since $S$ and $W$ are independent. Hence $X^q(T)$ cannot vanish almost surely conditional on $S=1$. Since $\vartheta(y)>0$ for $y\neq0$, it follows that $\mathcal M(q)>0$. The strict bounds on $q^*$ follow easily from \eqref{EqMFBoundaryFixedPoint}.
	
	For the flow in \eqref{EqMFBoundaryFlow}, the equality of the conditional moments and the fixed-point equation give
	\[
	\int\vartheta\,d\rho_1(T)
	=\int\vartheta\,d\rho_{-1}(T)
	=\mathcal M(q^*)=\frac{1-q^*}{\lambda}.
	\]
	Thus, \eqref{EqMFBoundaryMoment} holds with $q=q^*$.
	The informed values and controls in \eqref{EqMFBoundaryInformedValue}, together with Proposition~\ref{PropMFBoundaryStopping}, establish individual optimality and \eqref{EqMFBoundaryValues}.
	The prescribed strategy is admissible, since $\hat u$ and $q^*S$ are bounded and information about $S$ is used only after $\tau$.
	It produces $X^{q^*}$, so \eqref{EqMFBoundaryFlow} gives consistency by definition.
	The conditional laws are Borel measurable in time, since $X^{q^*}$ is continuous and $t\mapsto\mathbf{1}_{\{\tau\leq t\}}$ is right-continuous.
	Consequently, both conditions in Definition~\ref{DefMFGstrong} hold.
\end{proof}

The interaction determines the equilibrium informed drift $q^*S$, while the buying boundary remains the explicit curve $|x|=c\sigma\sqrt{T-t}$. Once $q=q^*$ is selected, the information price is fixed and does not change under individual deviations. Uniqueness of $q^*$ concerns the scalar consistency equation; no uniqueness claim is made for general equilibria of the resulting mean field game.

Figure~\ref{FigMFComparativeC} illustrates the comparative statics of the equilibrium drift with respect to the coupling strength $\lambda$ and the acquisition-boundary level $c$. We fix $T=\sigma=1$, $\varepsilon=0.12$, and $\nu=\frac{1}{2}(\delta_{-1.35}+\delta_{1.35})$, and for each pair $(\lambda,c)$ numerically evaluate the function $\mathcal M(q)$ in \eqref{EqMFBoundaryM} and solve the scalar consistency equation $q=1-\lambda\mathcal M(q)$. The left panel shows~$q^*$ as $c$ varies for three fixed values of $\lambda$, while the right panel shows $q^*$ as $\lambda$ varies for three fixed values of $c$. For the parameter ranges displayed, the equilibrium drift is only mildly sensitive to the acquisition-boundary level $c$, even though changing $c$ changes the buying region $|x|\leq c\sigma\sqrt{T-t}$ and hence the distribution of the acquisition time. By contrast, the dependence on the coupling strength $\lambda$ is pronounced: $q^*$ decreases substantially as the interaction becomes stronger. The three curves in the right panel remain close, consistent with the comparatively weak dependence on $c$ observed in the left panel. Thus, in this example the mean field interaction primarily affects the magnitude of the informed drift, whereas changes in the acquisition threshold have a substantially smaller effect on the equilibrium consistency parameter over the range considered. For each pair $(\lambda,c)$, the information price in \eqref{EqMFBoundaryPrice} is chosen with the corresponding fixed point $q=q^*$.

\providecommand{\MFQvsCDataFile}{mf_q_vs_c.dat}
\providecommand{\MFQvsLambdaDataFile}{mf_q_vs_lambda.dat}
\providecommand{\MFComparativeMetaFile}{mf_comparative_meta.tex}
\def\MFT{1}
\def\MFSigma{1}
\def\MFEpsilon{0.12}
\def\MFDelta{0.2}
\def\MFXZero{1.35}
\def\MFThetaScale{1}
\def\MFCompCMin{0.5}
\def\MFCompCMax{1.3}
\def\MFCompCStep{0.2}
\def\MFCompLambdaMax{0.6}
\def\MFCompLambdaStep{0.1}
\def\MFCompLeftYMin{0.635347}
\def\MFCompLeftYMax{0.903926}
\def\MFCompRightYMin{0.628177}
\def\MFCompRightYMax{1.027542}
\def\MFLambdaOne{0.2}
\def\MFLambdaTwo{0.45}
\def\MFLambdaThree{0.6}
\def\MFCOne{0.6}
\def\MFCTwo{1}
\def\MFCThree{1.3}
\def\MFSweepPairs{8000}
\def\MFSweepRepresentedPaths{16000}
\def\MFSweepSteps{500}
\def\MFSeed{20261014}

\begin{figure}[htbp]
	\centering
	\begin{subfigure}[t]{0.48\textwidth}
		\centering
		\begin{tikzpicture}[>=Latex]
			\begin{axis}[
				width=\linewidth,
				height=0.68\linewidth,
				xmin=\MFCompCMin,
				xmax=\MFCompCMax,
				ymin=\MFCompLeftYMin,
				ymax=\MFCompLeftYMax,
				axis x line=bottom,
				axis y line=left,
				axis line style={-Latex},
				xtick distance=\MFCompCStep,
				ytick distance=0.1,
				tick label style={font=\scriptsize},
				ymajorgrids=true,
				grid style={gray!14},
				legend style={draw=none,fill=none,font=\scriptsize,at={(0.97,0.92)},anchor=north east,row sep=-1pt},
				clip=false
				]
				\addplot[black,very thick] table[x=c,y=q_lambda_1]{\MFQvsCDataFile};
				\addlegendentry{$\lambda=\MFLambdaOne$}
				\addplot[black,thick,densely dashed] table[x=c,y=q_lambda_2]{\MFQvsCDataFile};
				\addlegendentry{$\lambda=\MFLambdaTwo$}
				\addplot[black,thick,densely dotted] table[x=c,y=q_lambda_3]{\MFQvsCDataFile};
				\addlegendentry{$\lambda=\MFLambdaThree$}
				\node[anchor=west,font=\normalsize] at (rel axis cs:1.015,0) {$c$};
				\node[anchor=south west,font=\normalsize] at (rel axis cs:-0.005,1.015) {$q^*$};
			\end{axis}
		\end{tikzpicture}
		\caption{$q^*$ as the acquisition-boundary level $c$ varies, for three fixed coupling strengths.}
		\label{FigMFComparativeC}
	\end{subfigure}
	\hfill
	\begin{subfigure}[t]{0.48\textwidth}
		\centering
		\begin{tikzpicture}[>=Latex]
			\begin{axis}[
				width=\linewidth,
				height=0.68\linewidth,
				xmin=0,
				xmax=\MFCompLambdaMax,
				ymin=\MFCompRightYMin,
				ymax=\MFCompRightYMax,
				axis x line=bottom,
				axis y line=left,
				axis line style={-Latex},
				xtick distance=\MFCompLambdaStep,
				ytick distance=0.1,
				tick label style={font=\scriptsize},
				ymajorgrids=true,
				grid style={gray!14},
				legend style={draw=none,fill=none,font=\scriptsize,at={(0.03,0.03)},anchor=south west,row sep=-1pt},
				clip=false
				]
				\addplot[black,very thick] table[x=lambda,y=q_c_1]{\MFQvsLambdaDataFile};
				\addlegendentry{$c=\MFCOne$}
				\addplot[black,thick,densely dashed] table[x=lambda,y=q_c_2]{\MFQvsLambdaDataFile};
				\addlegendentry{$c=\MFCTwo$}
				\addplot[black,thick,densely dotted] table[x=lambda,y=q_c_3]{\MFQvsLambdaDataFile};
				\addlegendentry{$c=\MFCThree$}
				\node[anchor=west,font=\normalsize] at (rel axis cs:1.015,0) {$\lambda$};
				\node[anchor=south west,font=\normalsize] at (rel axis cs:-0.005,1.015) {$q^*$};
			\end{axis}
		\end{tikzpicture}
		\caption{$q^*$ as the coupling strength $\lambda$ varies, for three fixed acquisition-boundary levels.}
		\label{FigMFComparativeLambda}
	\end{subfigure}
	\caption{Comparative statics of the equilibrium informed drift. The remaining parameters are $T=\MFT$, $\sigma=\MFSigma$, $\varepsilon=\MFEpsilon$, initial positions $\pm\MFXZero$, and $\vartheta(y)=y^2/(\MFThetaScale+y^2)$.}
	\label{FigMFComparativeStatics}
\end{figure}


\section{Compatible \texorpdfstring{$N$}{N}-player games}\label{SectApproximation}

In this section, we introduce $N$-player games that are compatible with the mean field game studied above in the sense that (strong) solutions to the mean field game induce approximate $N$-player Nash equilibria with approximation error tending to zero as $N$ goes to infinity. The information structure of the $N$-player game will mimic that of the mean field game. For the construction of approximate $N$-player Nash equilibria, it will be convenient to define all $N$-player games on the same probability space.

Let $(\hat{\Omega},\hat{\cF}, \hat{\Prb})$ be a complete probability space with $\hat{\bF}=\{\hat{\cF}_t\}_{t\ge0}$ a complete filtration in $\hat{\cF}$ carrying $d$-dimensional $\hat\bF$-Wiener processes $W_{1},W_{2},\ldots$ starting in zero, an $\hat\cF_0$-measurable $\bS$-valued random variable $\hat{S}$ with distribution $\mu$, and identically distributed $\hat\cF_0$-measurable $\bR^d$-valued random variables $\xi_{1}, \xi_{2},\ldots$ with common distribution $\nu$ such that
\[
\hat{S}, \xi_{1}, \xi_{2},\ldots, W_{1},W_{2},\ldots \text{ are all independent}.
\]

For $N\in \bN\setminus\{1\}$, let $\boldsymbol{W}^{N}$ denote the $N\times d$-dimensional Wiener process $(W_{1},\ldots,W_{N})$, let $\boldsymbol{\xi}^{N}$ denote the $\bR^{N\times d}$-valued random variable $(\xi_{1},\ldots,\xi_{N})$, and let $\bF^{\boldsymbol{\xi},\boldsymbol{W}}_{N}$ denote the filtration generated by $\boldsymbol{\xi}^{N}$ and $\boldsymbol{W}^{N}$:
\[
\cF^{\boldsymbol{\xi},\boldsymbol{W}}_{t,N} \doteq \boldsymbol{\sigma}\left( \xi_{i}, W_{i}(r) : r \in [0,t],\; i\in \{1,\ldots,N\} \right), \quad t \geq 0.
\]
Let $\cT^{\boldsymbol{\xi}}_{T,N}$ denote the set of all $[0,T]$-valued $\bF^{\boldsymbol{\xi},\boldsymbol{W}}_{N}$-stopping times. For $\tau \in \cT^{\boldsymbol{\xi}}_{T,N}$, let $\bF^{\boldsymbol{\xi},\boldsymbol{W},\hat S(\tau)}_{N}$ denote the filtration given by
\[
\cF^{\boldsymbol{\xi},\boldsymbol{W},\hat{S}(\tau)}_{t,N} \doteq 
\boldsymbol{\sigma}\left(\xi_{i}, W_{i}(r), \hat{S}\cdot \mathbf{1}_{[\tau,\infty)}(r) : r\in [0,t],\; i\in \{1,\ldots,N\} \right), \quad t\geq 0.
\]
	Let $\cA^{\boldsymbol{\xi}}_{\tau,N}$ denote the set of all square-integrable $\Gamma$-valued $\bF^{\boldsymbol{\xi},\boldsymbol{W},\hat{S}(\tau)}_{N}$-progressively measurable processes.
	Given a vector of stopping times $\boldsymbol{\tau} = (\tau_{1},\ldots,\tau_{N}) \in \cT^{\boldsymbol{\xi},N}_{T,N}\doteq \times^{N} \cT^{\boldsymbol{\xi}}_{T,N}$, let $\cA^{\boldsymbol{\xi},N}_{\boldsymbol{\tau},N} \doteq \times_{i=1}^{N} \cA^{\boldsymbol{\xi}}_{\tau_{i},N}$ be the set of all strategy vectors $\boldsymbol{\alpha} = (\alpha_{1},\ldots,\alpha_{N})$ such that $\alpha_{i}\in \cA^{\boldsymbol{\xi}}_{\tau_{i},N}$ for every $i\in \{1,\ldots,N\}$.
	
	For $\boldsymbol{\tau}\in \cT^{\boldsymbol{\xi},N}_{T,N}$, $\boldsymbol{\alpha}\in \cA^{\boldsymbol{\xi},N}_{\boldsymbol{\tau},N}$, define the expected costs of player $i\in \{1,\ldots,N\}$ when players' positions start from $\boldsymbol{\xi}^{N}$ at time zero according to
	\[
	J^{N}_{i}(\boldsymbol{\tau},\boldsymbol{\alpha}) \doteq \hat{\Mean}\left[ \int_{0}^{T} \! f\left(t,X^{\alpha_{i}}_{i}(t),\hat{S},\alpha_{i}(t),\mu^{N}_{i}(t)\right)dt + h\bigl(\tau_{i}, X^{\alpha_{i}}_{i}(\tau_{i})\bigr) + g\left(X^{\alpha_{i}}_{i}(T),\hat{S},\mu_{i}^{N}(T)\right) \right],
	\]
	where $\hat{\Mean}$ denotes expectation with respect to $\hat{\Prb}$, the position of player $i$ is given by
	\begin{equation} \label{EqDynamicsNPl}
		X^{\alpha_{i}}_{i}(t) = \xi_{i} + \int_{0}^{t} \alpha_{i}(r)dr + \sigma W_{i}(t),\quad t\geq 0,
	\end{equation}
	and $\mu^{N}_{i}(t)$ is the empirical measure of the positions and information states of the other players at time $t$:
	\begin{equation} \label{ExNPlayerEmp}
		\mu^{N}_{i}(t) \doteq \frac{1}{N-1} \sum_{j\neq i} \delta_{(X^{\alpha_{j}}_{j}(t),\, \mathbf{1}_{[\tau_{j},\infty)}(t))},\quad t\in [0,T].
	\end{equation}
	Here, the position vector $\boldsymbol{X}^{N,\boldsymbol{\alpha}} = (X^{\alpha_{1}}_{1},\ldots,X^{\alpha_{N}}_{N})$ of all $N$ players evolves according to 
	\begin{equation} \label{EqDynamicsNPlvec}
		\boldsymbol{X}^{N,\boldsymbol{\alpha}}(t) \doteq \boldsymbol{\xi}^{N} + \int_{0}^{t} \boldsymbol{\alpha}(r)dr + \sigma\cdot \boldsymbol{W}^{N}(t), \quad t\geq 0.
	\end{equation}
	As the random vector of initial positions will be fixed, we will omit the dependence on $\boldsymbol{\xi}^{N}$ from the notation of players' positions.
	
	In this set-up, approximate Nash equilibria for the $N$-player game with the option to buy information are defined as follows.
	
	\begin{definition} \label{DefNplayerNash}
		Let $\epsilon \geq 0$. A pair $(\boldsymbol{\tau}, \boldsymbol{\alpha})$ with $\boldsymbol{\tau} \in \mathcal{T}^{\boldsymbol{\xi},N}_{T,N}$ and $\boldsymbol{\alpha} \in \mathcal{A}^{\boldsymbol{\xi},N}_{\boldsymbol{\tau},N}$ is called an \emph{$\epsilon$-Nash equilibrium} for the $N$-player game starting in $\boldsymbol{\xi}^{N}$ at time zero
		if for all $i\in \{1,\ldots,N\}$,
		\[
		J^{N}_{i}(\boldsymbol{\tau},\boldsymbol{\alpha}) \leq \epsilon + \inf_{\tilde{\tau}\in \mathcal{T}^{\boldsymbol{\xi}}_{T,N}} \inf_{\tilde{\alpha}\in \mathcal{A}^{\boldsymbol{\xi}}_{\tilde{\tau},N}} J^{N}_{i}\bigl( (\boldsymbol{\tau}^{-i}, \tilde{\tau}), (\boldsymbol{\alpha}^{-i},\tilde{\alpha})\bigr)
		\]
		with the usual notation for substitution of vector components. If the above inequality holds with $\epsilon = 0$, then $(\boldsymbol{\tau}, \boldsymbol{\alpha})$ is simply called a \emph{Nash equilibrium} for the $N$-player game.
	\end{definition}
	
	Approximate $N$-player Nash equilibria can be constructed starting from a strong solution of the mean field game.
	
	\begin{thrm} \label{ThApproxNash}
		Grant \hyprefall. Let $\rho$ be a conditional flow of measures. If $\rho$ is a strong solution in the sense of Definition~\ref{DefMFGstrong}, then there exist $(\boldsymbol{\tau}^{N}, \boldsymbol{\alpha}^{N})$ and $\epsilon_{N}\in [0,\infty)$, $N\in \bN\setminus\{1\}$, such that
		\begin{itemize}
			\item for every $N$, $J^{N}_{1}(\boldsymbol{\tau}^{N},\boldsymbol{\alpha}^{N}) = J^{N}_{i}(\boldsymbol{\tau}^{N},\boldsymbol{\alpha}^{N})$ for all $i\in\{1,\ldots,N\}$, and $(\boldsymbol{\tau}^{N}, \boldsymbol{\alpha}^{N})$ is an $\epsilon_{N}$-Nash equilibrium for the $N$-player game;
			
			\item $\lim_{N\to\infty} \epsilon_{N} = 0$ and $\lim_{N\to\infty} J^{N}_{1}(\boldsymbol{\tau}^{N},\boldsymbol{\alpha}^{N}) = \inf_{\tilde{\tau}\in \mathcal{T}^{\xi}_{T}} \inf_{\tilde{\alpha}\in \mathcal{A}^{\xi}_{\tilde{\tau}}} J_{\rho}(\tilde{\tau},\tilde{\alpha})$.
		\end{itemize}
	\end{thrm}
	
	\begin{proof}
		Let $\rho$ be a strong solution in the sense of Definition~\ref{DefMFGstrong}. Choose $\tau^{\ast}\in \mathcal{T}^{\xi}_{T}$, $\alpha^{\ast}\in \mathcal{A}^{\xi}_{\tau^{\ast}}$ such that
		\begin{align*}
			J_{\rho}(\tau^{\ast},\alpha^{\ast}) &\leq \inf_{\tilde{\tau}\in \mathcal{T}^{\xi}_{T}} \inf_{\tilde{\alpha}\in \mathcal{A}^{\xi}_{\tilde{\tau}}} J_{\rho}(\tilde{\tau},\tilde{\alpha}) \\
			\intertext{and, for $\mu$-almost every $s\in \bS$,}
			\rho_{s}(t) &= \Prb_{s}\circ \left( X^{\alpha^{\ast}}(t), \mathbf{1}_{[\tau^{\ast},\infty)}(t) \right)^{-1} \text{ for all }t\in [0,T],
		\end{align*}
		where $X^{\alpha^{\ast}}$ is given by \eqref{EqDynamicsXi} and $\Prb_{s}(.) \doteq \Prb(\,.\mid S=s)$ denotes conditional probability under $\Prb$ given $S = s$. By definition of $\mathcal{T}^{\xi}_{T}$ and $\mathcal{A}^{\xi}_{\tau^{\ast}}$, respectively, and thanks to Doob's functional representation, we can find Borel measurable progressive functions
		\begin{align*}
			& \mathfrak{z}\colon [0,\infty)\times \bR^{d} \times \mathbf{C}([0,\infty),\bR^{d}) \rightarrow \{0,1\}, & & \mathfrak{a}\colon [0,\infty)\times \bR^{d} \times \mathbf{C}([0,\infty),\bR^{d})\times \bS \rightarrow \Gamma &
		\end{align*}
		such that, for $\Prb$-almost every $\omega\in\Omega$ and all $t\geq 0$,
		\begin{align*}
			& \mathbf{1}_{[\tau^{\ast}(\omega),\infty)}(t) = \mathfrak{z}\bigl(t,\xi(\omega),W(.,\omega)\bigr), &  &\alpha^{\ast}(t,\omega) = \mathfrak{a}\bigl(t,\xi(\omega),W(.,\omega), S(\omega)\cdot \mathbf{1}_{[\tau^{\ast}(\omega),\infty)}(t)\bigr), &
		\end{align*} 
		where we interpret the control process $\alpha^{\ast}$ as a random variable taking values in the space of $\Gamma$-valued deterministic relaxed controls over $[0,\infty)$. For every $i\in \bN$, define random elements on $(\hat{\Omega},\hat{\cF}, \hat{\Prb})$ according to
		\begin{align*}
			&\tau_{i}(\omega) \doteq T \wedge \inf\{t\in [0,T] :  \mathfrak{z}\bigl(t,\xi_{i}(\omega),W_{i}(.,\omega)\bigr) = 1\},& && && \\
			&\alpha_{i}(t,\omega) \doteq \mathfrak{a}\bigl(t,\xi_{i}(\omega),W_{i}(.,\omega), \hat{S}(\omega)\cdot \mathbf{1}_{[\tau_{i}(\omega),\infty)}(t)\bigr),& &t\geq 0,& &\omega\in \hat{\Omega}.&
		\end{align*}
		Notice that $\tau_{i}\in \cT^{\boldsymbol{\xi}}_{T,N}$ and $\alpha_{i}\in \cA^{\boldsymbol{\xi}}_{\tau_{i},N}$ whenever $i\leq N$. Let $X_{i}^{\alpha_{i}}$ be given by \eqref{EqDynamicsNPl}. Then, by construction,
		\begin{equation} \label{EqApproxNashID}
			\hat{\Prb}\circ \left(\tau_{i},\alpha_{i},X^{\alpha_{i}}_{i},\hat{S}\right)^{-1} = \Prb\circ \left(\tau^{\ast}, \alpha^{\ast}, X^{\alpha^{\ast}}, S\right)^{-1} \text{ for all }i\in \bN,
		\end{equation}
		and the triplets $(\tau_{1},\alpha_{1},X^{\alpha_{1}}_{1}), (\tau_{2},\alpha_{2},X^{\alpha_{2}}_{2}),\ldots$ are conditionally independent given $\hat{S}$.
		
		As in \eqref{ExNPlayerEmp}, for $i\in \{1,\ldots,N\}$, let $\mu^{N}_{i}(.)$ be the associated flow of position and information state empirical measures:
		\[
		\mu^{N}_{i}(t,\omega) \doteq \frac{1}{N-1} \sum_{j\neq i} \delta_{(X^{\alpha_{j}}_{j}(t,\omega),\, \mathbf{1}_{[\tau_{j}(\omega),\infty)}(t))},\quad \omega\in \hat{\Omega}.
		\]
		In view of \eqref{EqApproxNashID} and by conditional independence, we deduce from Varadarajan's theorem \citep[cf.\ Theorem~11.4.1 in][p.\,399]{dudley02} that for $\mu$-almost every $s\in \bS$,
		\begin{equation} \label{EqApproxNashEmpConv}
			\hat{\Prb}_{s}\left( \left\{ \omega\in \hat{\Omega} : \mu^{N}_{i}(t,\omega) \stackrel{N\to\infty}{\longrightarrow} \rho_{s}(t) \right\} \right) = 1 \text{ for all }t\in [0,T],
		\end{equation}
		where $\hat{\Prb}_{s}(.) \doteq \hat{\Prb}(\,.\mid \hat S=s)$ denotes conditional probability under $\hat{\Prb}$ given $\hat{S} = s$.

		For $N\in \bN$, set
		\begin{align*}
			&\boldsymbol{\tau}^{N} \doteq (\tau_{1},\ldots,\tau_{N}),& &\boldsymbol{\alpha}^{N} \doteq (\alpha_{1},\ldots,\alpha_{N}),& &\boldsymbol{X}^{N,\boldsymbol{\alpha}^{N}} \doteq (X^{\alpha_{1}}_{1},\ldots,X^{\alpha_{N}}_{N}).&
		\end{align*}
		Notice that $\boldsymbol{\tau}^{N}\in \cT^{\boldsymbol{\xi},N}_{T,N}$, $\boldsymbol{\alpha}^{N}\in \cA^{\boldsymbol{\xi},N}_{\boldsymbol{\tau}^{N},N}$, and that $\boldsymbol{X}^{N,\boldsymbol{\alpha}^{N}}$ satisfies \eqref{EqDynamicsNPlvec}. Set
		\[
		\epsilon_{N}\doteq \max \Bigl\{ J^{N}_{i}(\boldsymbol{\tau}^{N},\boldsymbol{\alpha}^{N}) - \inf_{\tilde{\tau}\in \mathcal{T}^{\boldsymbol{\xi}}_{T,N}} \inf_{\tilde{\alpha}\in \mathcal{A}^{\boldsymbol{\xi}}_{\tilde{\tau},N}} J^{N}_{i}\bigl( (\boldsymbol{\tau}^{N,-i}, \tilde{\tau}), (\boldsymbol{\alpha}^{N,-i},\tilde{\alpha})\bigr) : i\in \{1,\ldots,N\} \Bigr\}.
		\]
		By symmetry of construction, we have $J^{N}_{i}(\boldsymbol{\tau}^{N},\boldsymbol{\alpha}^{N}) = J^{N}_{1}(\boldsymbol{\tau}^{N},\boldsymbol{\alpha}^{N})$ for all $i\in \{1,\ldots,N\}$ and
		\[
		\epsilon_{N} = J^{N}_{1}(\boldsymbol{\tau}^{N},\boldsymbol{\alpha}^{N}) - \inf_{\tilde{\tau}\in \mathcal{T}^{\boldsymbol{\xi}}_{T,N}} \inf_{\tilde{\alpha}\in \mathcal{A}^{\boldsymbol{\xi}}_{\tilde{\tau},N}} J^{N}_{1}\bigl( (\boldsymbol{\tau}^{N,-1}, \tilde{\tau}), (\boldsymbol{\alpha}^{N,-1},\tilde{\alpha})\bigr).
		\]
		We can therefore concentrate on the costs and deviations from equilibrium of player one. Choose $\tilde{\tau}_{1}^{N}\in \mathcal{T}^{\boldsymbol{\xi}}_{T,N}$ and $\tilde{\alpha}_{1}^{N}\in \mathcal{A}^{\boldsymbol{\xi}}_{\tilde{\tau}_{1}^{N},N}$ such that
		\[
		J^{N}_{1}\bigl( (\boldsymbol{\tau}^{N,-1}, \tilde{\tau}_{1}^{N}), (\boldsymbol{\alpha}^{N,-1},\tilde{\alpha}_{1}^{N})\bigr) \leq \frac{1}{N} + \inf_{\tilde{\tau}\in \mathcal{T}^{\boldsymbol{\xi}}_{T,N}} \inf_{\tilde{\alpha}\in \mathcal{A}^{\boldsymbol{\xi}}_{\tilde{\tau},N}} J^{N}_{1}\bigl( (\boldsymbol{\tau}^{N,-1}, \tilde{\tau}), (\boldsymbol{\alpha}^{N,-1},\tilde{\alpha})\bigr).
		\]
		To establish the theorem, it suffices to show the following:
		\begin{itemize}
			\item Convergence of near equilibrium costs: $\lim_{N\to\infty} J^{N}_{1}(\boldsymbol{\tau}^{N},\boldsymbol{\alpha}^{N}) = J_{\rho}(\tau^{\ast},\alpha^{\ast})$.
			
			\item Asymptotic optimality: $\liminf_{N\to\infty} J^{N}_{1}\bigl( (\boldsymbol{\tau}^{N,-1}, \tilde{\tau}_{1}^{N}), (\boldsymbol{\alpha}^{N,-1},\tilde{\alpha}_{1}^{N})\bigr) \geq J_{\rho}(\tau^{\ast},\alpha^{\ast})$.
			
		\end{itemize}
		
		\paragraph{Convergence of near equilibrium costs.} For $s\in \bS$, let $\Mean_{s}$, $\hat{\Mean}_{s}$ denote expectation with respect to $\Prb_{s}$ and $\hat{\Prb}_{s}$, respectively. Then, for any $N$, by Fubini's theorem and the choice of $(\boldsymbol{\tau}^{N},\boldsymbol{\alpha}^{N})$, we have
		\begin{multline*}
			J^{N}_{1}(\boldsymbol{\tau}^{N},\boldsymbol{\alpha}^{N}) = \int_{0}^{T}  \int_{\bS} \hat{\Mean}_{s}\left[f\left(t,X^{\alpha_{1}}_{1}(t),s,\alpha_{1}(t),\mu^{N}_{1}(t)\right)\right] \mu(ds)\, dt \\
			+ \int_{\bS} \hat{\Mean}_{s}\left[ h\bigl(\tau_{1}, X^{\alpha_{1}}_{1}(\tau_{1})\bigr) \right] \mu(ds) + \int_{\bS} \hat{\Mean}_{s}\left[ g\left(X^{\alpha_{1}}_{1}(T),s,\mu_{1}^{N}(T)\right) \right] \mu(ds).
		\end{multline*}
		Notice that the strategy, stopping time, and position of the first player, when she is compliant, do not depend on $N$. Thanks to \eqref{EqApproxNashEmpConv} and dominated convergence (using assumption \hypref{ARunTermCosts} on $f$) as well as \eqref{EqApproxNashID}, we obtain
		\begin{align*}
			\lim_{N\to\infty} &\int_{0}^{T}  \int_{\bS} \hat{\Mean}_{s}\left[f\left(t,X^{\alpha_{1}}_{1}(t),s,\alpha_{1}(t),\mu^{N}_{1}(t)\right)\right] \mu(ds)\, dt \\
			&=\int_{0}^{T}  \int_{\bS} \hat{\Mean}_{s}\left[f\left(t,X^{\alpha_{1}}_{1}(t),s,\alpha_{1}(t),\rho_{s}(t)\right)\right] \mu(ds)\, dt \\
			&= \int_{0}^{T}  \int_{\bS} \Mean_{s}\left[f\left(t,X^{\alpha^{\ast}}(t),s,\alpha^{\ast}(t),\rho_{s}(t)\right)\right] \mu(ds)\, dt,
		\end{align*}
		and, in the same way, this time using \hypref{ARunTermCosts} on $g$,
		\[
		\lim_{N\to\infty} \int_{\bS} \hat{\Mean}_{s}\left[ g\left(X^{\alpha_{1}}_{1}(T),s,\mu_{1}^{N}(T)\right) \right] \mu(ds) = \int_{\bS} \Mean_{s}\left[ g\left(X^{\alpha^{\ast}}(T),s,\rho_{s}(T)\right) \right] \mu(ds).
		\]
		Convergence of the terms involving $h$ trivially follows from \eqref{EqApproxNashID} since, by assumption, $h$ does not depend on the measure argument.
		
		Putting everything together, we find that
		\begin{align*}
			&\begin{aligned}
				\lim_{N\to\infty} J^{N}_{1}(\boldsymbol{\tau}^{N},\boldsymbol{\alpha}^{N}) &=
				\int_{0}^{T}  \int_{\bS} \Mean_{s}\left[f\left(t,X^{\alpha^{\ast}}(t),s,\alpha^{\ast}(t),\rho_{s}(t)\right)\right] \mu(ds)\, dt \\
				& + \int_{\bS} \Mean_{s}\left[ h\bigl(\tau^{\ast},X^{\alpha^{\ast}}(\tau^{\ast})\bigr) + g\left(X^{\alpha^{\ast}}(T),s,\rho_{s}(T)\right) \right] \mu(ds)
			\end{aligned}\\
			&= \Mean\left[ \int_{0}^{T} f\left(t,X^{\alpha^{\ast}}(t),S,\alpha^{\ast}(t),\rho_{S}(t)\right)dt + h\bigl(\tau^{\ast},X^{\alpha^{\ast}}(\tau^{\ast})\bigr) + g\left(X^{\alpha^{\ast}}(T),S,\rho_{S}(T)\right) \right] \\
			&= J_{\rho}(\tau^{\ast},\alpha^{\ast}).
		\end{align*}

		\paragraph{Asymptotic optimality.} From \eqref{EqDynamicsNPl}, by standard estimates, we have, for all $N\in \bN$,
		\begin{equation} \label{EqApproxNashPositionEst}
			\sup_{t\in[0,T]} \hat{\Mean}\left[ |X^{\tilde{\alpha}^{N}_{1}}_{1}(t)|^{2} \right] \leq 3\left( \hat{\Mean}\left[ |\xi_1|^{2} \right] + d\,\sigma^{2}T + T\hat{\Mean}\left[ \int_{0}^{T}|\tilde{\alpha}^{N}_{1}(r)|^{2}dr \right]\right).
		\end{equation}
		The same estimate holds true with $\hat{\Mean}$ replaced by $\hat{\Mean}_{s}$, $s\in \bS$ with $\mu(\{s\}) > 0$. The coercivity assumption \hypref{ACoercivity} together with the non-negativity of $h$ thus imply that the costs of the deviating player along any subsequence with average control energy tending to infinity will tend to infinity, too. We can therefore assume without loss of generality that
		\begin{equation} \label{EqApproxNashControlEnergy}
			\max_{s\in \bS : \mu(\{s\}) > 0} \sup_{N\in \bN} \hat{\Mean}_{s} \left[ \int_{0}^{T}|\tilde{\alpha}^{N}_{1}(r)|^{2}dr \right] < \infty.
		\end{equation}
		
		For notational convenience, we write $\tilde J^{N}_{1}$ for $J^{N}_{1}\bigl((\boldsymbol{\tau}^{N,-1},\tilde{\tau}^{N}_{1}),(\boldsymbol{\alpha}^{N,-1},\tilde{\alpha}^{N}_{1})\bigr)$.
		Then, by Fubini's theorem, for every $N$,
		\begin{multline*}
			\tilde J^{N}_{1} = \int_{0}^{T} \int_{\bS}\hat{\Mean}_{s}\left[f\left(t,X^{\tilde{\alpha}^{N}_{1}}_{1}(t), s,\tilde{\alpha}^{N}_{1}(t),\mu^{N}_{1}(t)\right)\right] \mu(ds)\, dt \\
			+ \int_{\bS} \hat{\Mean}_{s}\left[h\bigl(\tilde{\tau}^{N}_{1},X^{\tilde{\alpha}^{N}_{1}}_{1}(\tilde{\tau}^{N}_{1})\bigr)\right] \mu(ds) 
			+ \int_{\bS} \hat{\Mean}_{s}\left[g\left(X^{\tilde{\alpha}^{N}_{1}}_{1}(T),s,\mu_{1}^{N}(T)\right)\right] \mu(ds).
		\end{multline*}
		Notice that the flow of empirical measures $\mu_{1}^{N}$ is the same as before, as it only depends on the non-deviating players. Let $\hat{J}^{N}_{1}$ denote the corresponding cost of player one when the flow of measures $\mu^{N}_{1}$ is replaced by its mean field limit:
		\[
		\hat{J}^{N}_{1} \doteq \hat{\Mean}\left[ \int_{0}^{T} \! f\left(t,X^{\tilde{\alpha}^{N}_{1}}_{1}(t),\hat{S},\tilde{\alpha}^{N}_{1}(t),\rho_{\hat{S}}(t)\right)dt + h\bigl(\tilde{\tau}^{N}_{1}, X^{\tilde{\alpha}^{N}_{1}}_{1}(\tilde{\tau}^{N}_{1})\bigr) + g\left(X^{\tilde{\alpha}^{N}_{1}}_{1}(T),\hat{S},\rho_{\hat{S}}(T)\right) \right].
		\]
		Using again Fubini's theorem, the fact that $h$ does not depend on the measure argument, and assumption \hypref{ALocLipMeasure}, we have
		\begin{multline*}
			\left| \tilde J^{N}_{1} - \hat{J}^{N}_{1} \right| 
			\leq \int_{\bS} \hat{\Mean}_{s}\left[ L\left(1+ |X^{\tilde{\alpha}^{N}_{1}}_{1}(T)| \right) \mathrm{d}_{\cP}\left(\mu_{1}^{N}(T),\rho_{s}(T)\right) \right] \mu(ds) \\
			+ \int_{0}^{T}  \int_{\bS} \hat{\Mean}_{s}\left[ L\left(1+ |X^{\tilde{\alpha}^{N}_{1}}_{1}(t)| \right) \mathrm{d}_{\cP}\left(\mu_{1}^{N}(t),\rho_{s}(t)\right) \right] \mu(ds)\, dt.
		\end{multline*}
		Thanks to the Cauchy-Schwarz inequality together with \eqref{EqApproxNashPositionEst} and \eqref{EqApproxNashControlEnergy}, we deduce from \eqref{EqApproxNashEmpConv} and dominated convergence that
		\[
		\left| \tilde J^{N}_{1} - \hat{J}^{N}_{1} \right| \stackrel{N\to\infty}{\longrightarrow} 0.
		\]
		
		Now, we claim that $\hat{J}^{N}_{1} \geq J_{\rho}(\tau^{\ast},\alpha^{\ast})$ for all $N$. To see this, fix $N$. Let $\kappa$ denote (a version of) the regular conditional distribution (under $\hat{\Prb}$) of $(\tilde{\tau}^{N}_{1}, \hat{S}, \xi_{1}, \tilde{\alpha}^{N}_{1},X^{\tilde{\alpha}^{N}_{1}}_{1},W_{1})$ given $\xi_{2},\ldots,\xi_{N}$, $W_{2},\ldots,W_{N}$. 
		Then for $\hat{\Prb}\circ (\xi_{2},\ldots,\xi_{N}, W_{2},\ldots,W_{N})^{-1}$-almost every $(\boldsymbol{\xi},\boldsymbol{W}) \in \bR^{d\times (N-1)} \times \mathbf{C}([0,T],\bR^{d})^{N-1}$, we have, with a slight abuse of notation:
		\begin{itemize}
			\item $W_{1}$ is an $\hat{\bF}$-Wiener process also under $\kappa_{(\boldsymbol{\xi},\boldsymbol{W})}$;
			
			\item $\hat{S}, \xi_{1}, W_{1}$ are independent under $\kappa_{(\boldsymbol{\xi},\boldsymbol{W})}$;
			
			\item with $\kappa_{(\boldsymbol{\xi},\boldsymbol{W})}$-probability one, $X^{\tilde{\alpha}^{N}_{1}}_{1}$ satisfies \eqref{EqDynamicsNPl} with $i=1$ and control process $\tilde{\alpha}^{N}_{1}$;
			
			\item $\tilde{\tau}^{N}_{1}$ coincides up to a set of $\kappa_{(\boldsymbol{\xi},\boldsymbol{W})}$-measure zero with a stopping time with respect to the filtration generated by $\xi_{1}$ and $W_{1}$;
			
			\item $\tilde{\alpha}^{N}_{1}$ coincides up to a set of $\kappa_{(\boldsymbol{\xi},\boldsymbol{W})}$-measure zero with a control process that is progressively measurable with respect to the filtration generated by $\xi_{1}$, $W_{1}$, and $\hat{S}\cdot \mathbf{1}_{[\tilde{\tau}^{N}_{1},\infty)}$.
		\end{itemize}
		The last two points above follow from Doob's functional representation and the fact that $\tilde{\tau}_{1}^{N}\in \mathcal{T}^{\boldsymbol{\xi}}_{T,N}$ and $\tilde{\alpha}_{1}^{N}\in \mathcal{A}^{\boldsymbol{\xi}}_{\tilde{\tau}_{1}^{N},N}$.
		Indeed, for almost every fixed realization of the other players' randomness, the stopping time and control admit non-anticipative measurable representations in terms of the first player's initial position, Wiener process, and, for the control, the hidden state revealed at the stopping time. Applying these same representations to $(\xi,W,S)$ on the limit probability space preserves the stopping-time and progressive measurability properties.
		Since $(\xi,W,S)$ and $(\xi_1,W_1,\hat S)$ have the same joint law under the corresponding probability measures, the resulting processes have the required joint distribution.
		
		It follows that there exist $\tau\in \mathcal{T}^{\xi}_{T}$, $\alpha\in \mathcal{A}^{\xi}_{\tau}$ on the limit probability space such that
		\[
		\Prb\circ \left( \tau, S, \xi, \alpha, X^{\alpha}, W \right)^{-1} = \kappa_{(\boldsymbol{\xi},\boldsymbol{W})}.
		\]
		But this entails that $\hat{J}^{N}_{1} \geq J_{\rho}(\tau^{\ast},\alpha^{\ast})$ by the optimality of $(\tau^{\ast},\alpha^{\ast})$.
		
	\end{proof}
	

	\bibliographystyle{abbrvnat}

\end{document}